\documentclass[paper=a4]{scrartcl}
\pdfoutput=1
\usepackage[OT1]{fontenc}
\usepackage[utf8]{inputenc}
\usepackage[english]{babel}
\usepackage{amsmath, amssymb, amsthm, amsfonts, amsxtra}
\usepackage{enumerate,mathrsfs,hyperref,babelbib,graphicx}
\numberwithin{equation}{section}
\theoremstyle{plain}
\newtheorem{thm}{Theorem}[section]
\newtheorem{lem}[thm]{Lemma}
\newtheorem{prop}[thm]{Proposition}
\newtheorem{cor}[thm]{Corollary}
\theoremstyle{definition}

\DeclareMathOperator{\GL}{GL}
\newcommand{\tp}{^{\mathstrut\scriptscriptstyle\top}}
\DeclareMathOperator{\LB}{\Delta_{\mathit{B}}}
\DeclareMathOperator{\Aut}{Aut}
\DeclareMathOperator{\U}{U}
\DeclareMathOperator{\SU}{SU}
\DeclareMathOperator{\id}{id}

\DeclareMathOperator{\hc}{c}
\DeclareMathOperator{\supp}{supp}
\DeclareMathOperator{\osc}{osc}

\newcommand{\fup}[2]{#1^{\overline{#2}}}

\newcommand{\Q}{\mathbb{Q}}
\newcommand{\N}{\mathbb{N}}
\newcommand{\R}{\mathbb{R}}
\newcommand{\Z}{\mathbb{Z}}

\newcommand{\C}{\mathbb{C}}
\renewcommand{\L}{\mathit{L}}
\newcommand{\D}{\mathscr{D}}

\newcommand{\dd}{\mathrm{d}}

\renewcommand{\Re}{\operatorname{Re}}

\renewcommand{\setminus}{\smallsetminus}

\title{Time-Frequency Analysis in the Unit Ball} 
\author{Mathias Ionescu-Tira}
\date{}

\begin{document}
\maketitle
\begin{abstract}
  As an appropriate analog of the Euclidean short-time Fourier transform, we study a windowed version of the Helgason-Fourier transform on the complex unit ball and translate the theory of modulation/coorbit spaces. 
  As a result, atomic decompositions by means of Banach frames are obtained.
\end{abstract}



\section{Introduction}

The short-time Fourier transform (STFT) is a well established tool in harmonic analysis.
Originally introduced as a windowed Fourier transform
\begin{equation*}
  V_\psi f (t, \omega) = \int_\R f(s) e^{-i \omega s} \psi(s-t)\, \dd s
\end{equation*}
on the real line (the fixed window, or analyzing function, $\psi$ may be thought of as having small support or fast decay),
it allows for a unified study of the behavior of a function in time $t$ and frequency $\omega$.

The STFT becomes an object of representation theory of the Heisenberg group by means of translations and modulations, which has led to a vast number of generalizations.
In particular, every (unitary) representation
$\rho: G \to \mathcal{L}(\mathcal{H})$
of a locally compact group $G$ on a Hilbert space $\mathcal{H}$ gives rise to a \emph{voice transform} $V_\psi$ on $\mathcal{H}$ defined by
\begin{equation*}
  V_\psi f(x) = \left< f, \rho(x) \psi \right>.
\end{equation*}
If all $\rho(x) \psi$ ($x \in G$) belong to a closed subspace $M$ of $\mathcal{H}$, then $V_\psi$ extends to distributions in the dual $M'$; the right side is then interpreted as a dual product.

A central part of this theory is concerned with the discretization of this transform, aiming at decompositions of the type
\begin{equation*}
  f = \sum_i \left< f, e_i \right> e_i
\end{equation*}
with respect to suitable atomic functions $e_i$, which in general do not constitute an orthogonal basis, but instead a \emph{Banach frame}, in the sense that
\begin{equation*}
  A \left\Vert f \right\Vert \le \left\Vert \left(\left< f, e_i \right> \right)_i \right\Vert \le A' \left\Vert f \right\Vert
\end{equation*}
holds with $A$ and $A'$ independent of $f$.
The norms involved are those on continuous and discrete versions of a \emph{coorbit space}, a function space consisting of distributions (respectively sequences), whose transforms (respectively their evaluations at discrete points) belong to a solid, translation invariant Banach space, typically some weighted Lebesgue space.

While this concept has been developed in a fairly general setting and has been successfully applied to cover homogeneous spaces, the starting point is often a unitary group representation, and the constructions rely on the group structure.
We refer to \cite{fei_groe1} for a general treatise, and \cite{da_st_te1}, \cite{da_st_te2}, \cite{da_fo_rau_st_te} for the case of homogeneous spaces and quotients.

\bigskip

In the present case, we take a more geometric approach and construct a voice transform on a hyperbolic manifold -- the complex ball -- by means of translations and modulations.
The former are given by automorphisms (Möbius transformations) of the manifold, while the latter are provided by the kernel of the Fourier transform in the sense of Helgason \cite{helg_groups}, \cite{helg_geom}.
In absence of a related group representation, the algebraic tools in this context are replaced by suitable regularity and decay properties.
These allow for a discretization, or sampling, of this transform similar as the above.

\bigskip

We restrict ourselves to this basic setting, i.e., we do not deal with several possible generalizations, such as the case of weighted measures \cite{zhang}.
We refer to \cite{liu_peng} and \cite{ferreira_moebgyro} for corresponding generalizations of the Fourier transform in the case of the ball in \emph{real} $n$-dimensional space.

An approach to sampling based on eigenfunctions of differential operators, which also covers the ball, can be found in \cite{pesenson_average_samp} and \cite{fei_pes_fue}.

\bigskip

This work is structured as follows.
Section \ref{sct:preliminaries} lists some known facts about the hyperbolic geometry of the ball and introduces the Helgason-Fourier transform.

Section \ref{sct:timefreq} is concerned with the voice transform, for which the Euclidean STFT serves as a template.
Two versions of coorbit spaces are introduced, one replicating modulation spaces (subspaces of tempered distributions), the other aiming at reconstructing the more general setting of group representations.
The inversion formula is derived on both versions of these spaces.

In the last section \ref{sct:frames} we obtain two atomic decompositions on coorbit spaces, and examine the conditions of those leading to Banach frames.

\section{Preliminaries}
\label{sct:preliminaries}

In this section we establish the basic setting and notation.
Most results can be found in \cite{rudin_ball} and \cite{helg_groups}.

Throughout this article, $B = \{ z = (z_1, \dots, z_n)\tp \in \C^n : \left| z \right| < 1 \}$ will denote the open unit ball in $n$-dimensional complex space with origin $o$.
Let $\Aut(B)$ denote the group of biholomorphic functions on $B$, and let $\U(n)$ be the subgroup of unitary mappings restricted to $B$.
Then $\Aut(B)$ and the group
\begin{equation*}
  \SU(n,1) = \{ A \in \GL(n+1,\C) : \det A = 1 \text{ and } A^* J A = J \}
\end{equation*}
of unitary matrices leaving invariant the quadratic form
\begin{equation*}
  x = (x_1, \dots, x_{n+1}) \mapsto \bar{x}\tp J x = \left|x_1 \right|^2 + \ldots + \left| x_n \right|^2 - \left| x_{n+1} \right|^2,
\end{equation*}
are isomorphic.
Writing every $\varphi \in \SU(n,1)$ as a block matrix
\begin{equation*}
  \varphi = \begin{pmatrix} Q & b \\ c\tp & d \end{pmatrix}
\end{equation*}
with $Q \in \C^{n \times n}$ and $b,c \in \C^n$, this group acts on $B$ by means of Möbius transformations
\begin{equation*}
  z \mapsto \varphi(z) = \frac{Qz + b}{d + \left< z, c \right>},
\end{equation*}
where $\left< z, c \right> = z_1 \overline{c_1} + \dots + z_n \overline{c_n}$.
Let $a \in B$, $a \ne o$, and set $Q_a = \left( s_a - 1 \right) a\, \bar{a}\tp /|a|^2 - s_aI_n$, where $s_a = (1 - |a|^2)^{1/2}$ and $I_n$ is the $n$-size identity matrix.
Defining
\begin{equation*}
  \varphi_o = -J = -\begin{pmatrix} I_n & 0 \\ 0 & -1 \end{pmatrix}, \quad \varphi_a = \begin{pmatrix} Q_a & a \\ -a\tp & 1 \end{pmatrix} \quad (a \ne o),
\end{equation*}
every $\varphi \in \Aut(B)$ can be written (up to a normalizing factor) as $\varphi = u \circ \varphi_a$, where $a \in B$ and $u \in \U(n)$.
The mappings $\varphi_a$ take the role of translations on $B$, and possess the following properties.

\begin{thm}
  \label{thm:moeb_basics}
  For each $a \in B$ the following holds true:
\begin{enumerate}
  \item For all $z, w \in \overline{B}$,
    \begin{equation}
      1 - \left< \varphi_a(z), \varphi_a(w) \right>  = \frac{\left( 1 - \left< a, a \right>  \right) \left( 1 - \left< z, w \right>  \right)}{\left( 1 - \left< z, a \right>  \right) \left( 1 - \left< a, w \right>  \right)}.
    \label{eq:moeb_prod}
    \end{equation}
  \item $\varphi_a$ is an involution on $B$, and a homeomorphism $\overline{B} \to \overline{B}$.
\end{enumerate}
\end{thm}

With $\U(n) \subset \Aut(B)$ being the stabilizer subgroup of the origin $o \in B$, the ball can be viewed as the homogeneous space $\Aut(B) / \U(n)$.
Let $l$ be the Lebesgue measure on $\C^n$, normalized so that $l(B) = 1$, then the measure $\mu$ given by
\begin{equation*}
  \dd\mu(z) = \bigl( 1-|z|^2 \bigr)^{-n-1} \dd l(z)
\end{equation*}
is invariant under the action of $\Aut(B)$ and, up to constants, unique with this property.
Let $\sigma$ be the invariant Borel measure on the boundary sphere $S = \partial B$ with $\sigma(S) = 1$.
The Haar measures on $G = \Aut(B)$ and $K = \U(n)$ can then be normalized so that the integration with respect to $\mu$ and $\sigma$ becomes
\begin{equation*}
  \int_G f(g(o))\, \dd g = \int_B f(z)\, \dd\mu(z) \qquad \text{and} \qquad
  \int_K h(k\eta)\, \dd k = \int_S h(\zeta)\, \dd\sigma(\zeta).
\end{equation*}

Let $\Delta = 4 \sum_j \partial_j \overline{\partial}_j$ denote the Laplacian on $\C^n$ (set $\partial_j = \partial/\partial z_j$ and $\overline{\partial}_j = \partial/\partial \bar{z}_j$), then the invariant Laplacian $\LB$ on $B$ is given by
\begin{equation*}
  \LB f (z) = \Delta (f \circ \varphi_z)(o), \qquad f \in C^2(B)
\end{equation*}
or, more explicitly,
\begin{equation*}
  \LB f(z) = 4 \bigl( 1 - |z|^2 \bigr) \sum_{j,k} \bigl( \delta_{j,k} - z_j \bar{z}_k \partial_j \overline{\partial}_k \bigr) f(z).
\end{equation*}
This operator commutes with the action of $\Aut(B)$ and is self adjoint
(Green's identity for a wider class of operators $\Delta_{\alpha,\beta}$ is studied in \cite{ahern_bru_casc}, here $\LB = \Delta_{0,0}$ is a special member).

Like in the Euclidean case, the Fourier transform on the ball relies on (formal) eigenfunctions of the Laplace-Beltrami operator.
For fixed $\lambda \in \C$, $\zeta \in S$, consider the functions
\begin{equation*}
  z \mapsto P_{\lambda,\zeta}(z) = \left( \frac{1 - |z|^2}{\bigl| 1 - \left< z, \zeta \right> \bigr|^2}\right)^{\frac{n+i\lambda}{2}}
\end{equation*}
for $z \in B$.
They generalize the notion of plane waves to the ball, being of the form
\begin{equation*}
  P_{\lambda,\zeta}(z) = e^{(n+i\lambda) \xi(z,\zeta)}
\end{equation*}
where $\xi(z,\zeta)$ denotes the hyperbolic distance of the origin to the sphere tangential to $S$ in $\zeta$ and passing through $z \in B$.
(these \emph{horospheres} are the analog of hyperplanes in Euclidean space.)
The particular choice of the exponent $(n+i\lambda)/2$ is justified below.

A function $f$ on $B$ is called \emph{radial}, if $f \circ u = f$ for every $u \in \U(n)$.
Function spaces consisting of radial functions are indicated by the subscript $\natural$.
The functions $P_{\lambda, \zeta}$ admit the following characterization of radial eigenfunctions of $\LB$ (or \emph{spherical functions}).

\begin{thm}
  \label{thm:sph_func_unique}
  \hspace{0ex}
  \begin{enumerate}[(i)]
    \item For every $\lambda \in \C$, $\zeta \in S$,
      \begin{equation*}
        \LB P_{\lambda, \zeta} = -(n^2 + \lambda^2) P_{\lambda, \zeta}.
      \end{equation*}
    \item Every radial $g \in C_\natural^2(B)$ which satisfies $\LB g = -(n^2 + \lambda^2) g$ is a constant multiple of the function
  \begin{equation*}
    \phi_\lambda(z) = \int_S P_{\lambda, \zeta}(z)\, \dd\sigma(\zeta).
  \end{equation*}
  \end{enumerate}
\end{thm}

As such, the transform
\begin{equation}
  \hat{f}(\lambda,\zeta) = \int_B f(z) P_{-\lambda, \zeta}(z) \, \dd\mu(z), \qquad f \in \D(B),
  \label{eq:def_hft}
\end{equation}
called the \emph{Helgason-Fourier transform}, is indeed a generalization of the Euclidean Fourier transform to the ball.
Here $\D(B)$ is the space of smooth functions on $B$ with compact support.
For radial functions $f \in \D_\natural(B)$ the integral in \eqref{eq:def_hft} is independent of $\zeta$ and reduces to the \emph{spherical transform}
\begin{equation*}
  \hat{f} (\lambda) = \int_B f(z) \phi_{-\lambda} (z) \, \dd\mu(z).
\end{equation*}
Note the symmetry $\phi_\lambda = \phi_{-\lambda}$, following from theorem \ref{thm:sph_func_unique} \textit{(ii)}, leading to $\hat{f}(\lambda) = \hat{f}(-\lambda)$.

The spherical functions possess a description as hypergeometric functions.
Indeed, using the orthogonality of the monomials $\zeta^\beta$ in $\L^2(S,\sigma)$, and setting $\alpha = (n+i\lambda)/2$, a computation shows (cf. \cite[Prop. 1.4.10]{rudin_ball})
\begin{equation}
  \phi_\lambda (z) = \left(1-|z|^2\right)^\alpha {}_2F_1 (\alpha, \alpha, n, |z|^2),
  \label{eq:spher_hypergeom}
\end{equation}
where
\begin{equation*}
  {}_2F_1\left( a, b, c, z \right) = \sum_{k=0}^\infty \frac{\fup{a}{k}\, \fup{b}{k}}{\fup{c}{k}\, k!}\, z^k, \qquad |z| < 1,
\end{equation*}
with $\fup{x}{k} = x (x+1) \cdots (x+k-1)$ denoting rising factorials.
Another representation of $\phi_\lambda$ is obtained through a series expansion,
\begin{equation*}
  \phi_\lambda (\tanh r) = \hc(\lambda) \sum_{j=0}^\infty \Gamma_j(\lambda) e^{(i\lambda-n-2j)r} + \hc(-\lambda) \sum_{j=0}^\infty \Gamma_j(-\lambda) e^{(-i\lambda-n-2j)r},
\end{equation*}
where the coefficients $\Gamma_j$ can be calculated recursively from the eigenvalue equation, yielding meromorphic functions in $\lambda$.
Furthermore,
\begin{equation*}
  \hc(\lambda) = \lim_{r \to \infty} \phi_\lambda(\tanh r) e^{(n-i\lambda)r}, \qquad \Re i \lambda > 0
\end{equation*}
is Harish-Chandra's $\hc$-function.
Using \eqref{eq:spher_hypergeom}, this limit can be calculated explicitly, since
\begin{equation*}
  \lim_{r \to \infty} \left( 1-\tanh(r)^2 \right)^{(n-i\lambda)/2} e^{(n-i\lambda)r} = 2^{n-i\lambda}
\end{equation*}
and
\begin{equation*}
{}_2F_1(a,b,c,1) = \frac{\Gamma(c) \Gamma(c-a-b)}{\Gamma(c-a) \Gamma(c-b)}, \qquad \Re(c-a-b) > 0,
\end{equation*}
through which
\begin{equation}
  \hc(\lambda) = 2^{n-i\lambda} \frac{\Gamma(n) \Gamma(i\lambda)}{\Gamma \bigl(\frac{n+i\lambda}{2}\bigr)^2}, \qquad i\lambda \notin \Z_{\le 0}
  \label{eq:def_hc}
\end{equation}
is obtained by analytic extension.

The inversion formula for the spherical transform can be derived using the general strategy in \cite{helg_groups}, \cite{rosenberg}.
As a result,
\begin{equation}
  f(z) = \frac{1}{2} \int_\R \hat{f}(\lambda) \phi_\lambda(z) \left| \hc (\lambda) \right|^{-2} \dd\lambda
  \label{eq:sph_inversion}
\end{equation}
holds for every $f \in \D_\natural(B)$.
This can be used to derive the inversion formula \eqref{eq:helg_inv} for the Helgason transform, yielding
\begin{equation*}
  f(z) = \frac{1}{2} \int_\R \int_S \hat{f}(\lambda, \zeta) P_{\lambda,\zeta} (z) \, \dd\sigma(\zeta) \left| \hc(\lambda) \right|^{-2} \dd\lambda.
\end{equation*}
Multiplying with $\overline{f(z)}$ and integrating against $\dd\mu(z)$, we obtain the Plancherel identity
\begin{equation}
  \begin{split}
    \int_B \left| f(z) \right|^2 \dd\mu(z)
    &= \frac{1}{2} \int_\R \int_S \bigl| \hat{f}(\lambda, \zeta) \bigr|^2 \left| \hc(\lambda) \right|^{-2}  \dd\sigma(\zeta) \, \dd\lambda \\
    &= \int_{\R^+} \int_S \bigl| \hat{f}(\lambda, \zeta) \bigr|^2 \left| \hc(\lambda) \right|^{-2}  \dd\sigma(\zeta) \, \dd\lambda.
  \end{split}
  \label{eq:helg_plancherel}
\end{equation}

To simplify notation, we set
\begin{equation*}
  \hat{B} = S \times \R \quad \text{and}\quad 
  \dd\nu(\lambda,\zeta) = \frac{1}{2} \left| \hc(\lambda) \right|^{-2} \dd\sigma(\zeta) \dd\lambda.
\end{equation*}
Following theorem \ref{thm:FT_isometry}, the Helgason transform is readily established as an isometry from $\L^2 (B,\mu)$ to $\L^2(\hat{B},\nu)$.

\section{Time-Frequency Analysis}
\label{sct:timefreq}

The aim of this section is to introduce a windowed Fourier transform based on the Helgason transform, by adapting the construction scheme for the Euclidean short-time Fourier transform, as seen in \cite{groechenig}.
First, we consider analogs of time-frequency shifts, which comprise the unitary representation of the Heisenberg group used for the Euclidean STFT.

In the present case, translations are generated by the points in $B$ and take the form $f \to f \circ \varphi_z$, while each point $b = (\lambda, \zeta)$ in $\hat{B}$ provides a modulation of the form $f \to P_b \cdot f = P_{\lambda,\zeta} \cdot f$.
Combined, these form the mapping
\begin{equation*}
  \rho: B \times \hat{B} \to \mathcal{L}(\mathcal{Y}), \quad
  \rho(z,b) f = P_{b} \cdot (f \circ \varphi_z),
\end{equation*}
which, for every $(z,b)$ in phase space $B \times \hat{B}$, yields a linear operator acting on an appropriate function space $\mathcal{Y}$ (yet to be determined).
For a fixed, nonzero window function $\psi$, we define the \emph{voice transform} $V_\psi$ by
\begin{equation*}
  V_\psi f (z,b) = \left< f, \rho(z,b) \psi \right>
  = \int_B f(w) P_{-\lambda,\zeta}(w) \overline{\psi (\varphi_z(w))} \, \dd\mu(w).
\end{equation*}
The space $\L_\natural^2(B,\mu)$ of radial $\psi \in \L^2(B,\mu)$ is a suitable reservoir of window functions for $V_\psi$ to be well-defined for all $f \in \L^2(B,\mu)$.

\begin{thm}
  \label{thm:orth_rel}
  The orthogonality relations
  \begin{equation}
    \left< V_{\psi_1} f_1, V_{\psi_2} f_2 \right>_{\L^2(\mu \otimes \nu)} = \overline{\left< \psi_1, \psi_2 \right>}_{\L^2(\mu)} \left< f_1, f_2 \right>_{\L^2(\mu)}
    \label{eq:orth_rel}
  \end{equation}
    hold for all $f_1, f_2 \in \L^2(B, \mu)$ and $\psi_1, \psi_2 \in \L_\natural^2(B,\mu)$.
\end{thm}

\begin{proof}
  First, we show \eqref{eq:orth_rel} for windows in $\D_\natural(B)$, this being a dense subspace of $\L_\natural^2(\mu)$.
  In this case, since
  \begin{equation*}
    \int_B \bigl| f(z) \overline{\psi (\varphi_w(z))} \bigr|^2 \dd\mu(z) \le \left\Vert \psi \right\Vert^2_{\infty} \left\Vert f \right\Vert^2_{\L^2(\mu)},
  \end{equation*}
  we have $f \cdot \overline{\psi \circ \varphi_w} \in \L^2(\mu)$ for every $f \in \L^2(\mu)$, $\psi \in \D_\natural(B)$, $w \in B$.
  Applying the Plancherel identity for $V_\psi f (w,b) = (f \cdot \overline{\psi \circ \varphi_w})\sphat\, (b)$ we obtain
  \begin{align*}
    \left< V_{\psi_1} f_1, V_{\psi_2} f_2 \right>
    &= \int_B \int_{\hat{B}} V_{\psi_1} f_1(w,b) \overline{V_{\psi_2} f_2 (w,b)}\, \dd\nu(b) \dd\mu(w) \\
    &= \int_B \int_{\hat{B}} (f_1 \cdot \overline{\psi_1 \circ \varphi_w})\sphat\, (b) \overline{(f_2 \cdot \overline{\psi_2 \circ \varphi_w})\sphat\, (b)}\, \dd\nu(b) \dd\mu(w) \\
    &= \int_B \int_B f_1(z) \overline{\psi_1(\varphi_w(z))}\, \overline{f_2(z)} \psi_2(\varphi_w(z))\, \dd\mu(z) \dd\mu(w) \\
    &= \int_B f_1(z) \overline{f_2(z)} \int_B \overline{\psi_1(\varphi_z(w))} \psi_2(\varphi_z(w))\, \dd\mu(z) \dd\mu(w) \\
    &= \int_B f_1(z) \overline{f_2(z)} \int_B \overline{\psi_1(w)} \psi_2(w)\, \dd\mu(w) \dd\mu(z) \\
    &= \left< f_1, f_2 \right> \overline{\left< \psi_1, \psi_2 \right>}.
  \end{align*}
  Here we have used $\left| \varphi_z(w) \right| = \left| \varphi_w(z) \right|$ (cf. \eqref{eq:moeb_prod}) and the symmetry of $\psi_1, \psi_2$.
  Finally, by first extending the operator
  \begin{equation*}
    \psi_1 \mapsto \left< V_{\psi_1} f_1, V_{\psi_2} f_2 \right>
  \end{equation*}
  to $\psi_1 \in \L_\natural^2(B,\mu)$ while keeping $\psi_2 \in \D_\natural(B)$ fixed, and then
  \begin{equation*}
    \psi_2 \mapsto \left< V_{\psi_1} f_1, V_{\psi_2} f_2 \right>
  \end{equation*}
  to $\psi_2 \in \L_\natural^2(B,\mu)$, with $\psi_1 \in \L_\natural^2(B,\mu)$ fixed, the equality is proved to hold in general.
\end{proof}

Theorem \ref{thm:orth_rel} provides the following inversion formula.

\begin{thm}
  Let $\psi, \gamma \in \L^2_\natural(B,\mu)$ be two windows satisfying $\left< \psi, \gamma \right> \ne 0$.
  Then for every $f \in \L^2(B,\mu)$ the identity
  \begin{equation*}
    f = \left< \gamma, \psi \right>^{-1} \int_B \int_{\hat{B}} V_\psi f (z,b) \rho(z,b) \gamma \, \dd\nu(b)\dd\mu(z)
  \end{equation*}
  holds in the weak sense.
\end{thm}

\begin{proof}
  The vector-valued integral
  \begin{equation*}
    \tilde{f} = \left< \gamma, \psi \right>^{-1} \iint_{B \times \hat{B}} V_\psi f (z,b) \rho(z,b) \gamma \, \dd\mu(z) \, \dd\nu(b)
  \end{equation*}
  is well defined on $\L^2(B,\mu)$ and by theorem \ref{thm:orth_rel} satisfies
  \begin{align*}
    \bigl< \tilde{f}, h \bigr>
    &= \left< \gamma, \psi \right>^{-1} \int_{\hat{B}} \int_B  V_\psi f (z,b) \overline{ \left< h, \rho(z,b) \gamma \right>}\, \dd\mu(z)\, \dd\nu(b) \\
    &= \left< \gamma, \psi \right>^{-1} \bigl< V_\psi f, V_\gamma h \bigr> \\
    &= \left< \gamma, \psi \right>^{-1} \left< f, h \right> \left< \gamma, \psi \right> \\
    &= \left< f, h \right>
  \end{align*}
  for all $h \in \L^2(B,\mu)$. This implies $\tilde{f} = f$.
\end{proof}

\subsection{Weights and Distributions}

We take a parallel approach to \cite{groechenig} and extend the voice transform to tempered distributions.
This requires the use of weight functions on $B$ and $\hat{B}$.

A weight is an almost everywhere positive and measurable function, but there is no loss in considering solely continuous weights.
On $B$, such a function $\kappa$ is called submultiplicative, if
\begin{equation*}
  \kappa(\varphi_w(z)) \le \kappa(w) \kappa(z)
\end{equation*}
holds for all $z, w \in B$.
On $\hat{B}$, we call a weight $v$ radial, if $v(\lambda,\zeta) = v(|\lambda|)$ and submultiplicative, if in addition
\begin{equation*}
  v(\lambda + \eta) \le v(\lambda) v(\eta)
\end{equation*}
holds for all $\lambda, \eta \in \R$.
Note that $\kappa, v \ge 1$, if both are continuous and submultiplicative.
If, in addition, both are radial, then for any permutation $(i,j,k)$ of $\{ 1, 2, 3 \}$,
\begin{equation}
  \kappa(z_i) \le \kappa(z_j) \kappa(z_k)
  \label{eq:kappa_perm}
\end{equation}
holds whenever $z_1 = \varphi_{z_2}(z_3)$, and similarly
\begin{equation}
  v(\lambda_i) \le v(\lambda_j) v(\lambda_k),
  \label{eq:v_perm}
\end{equation}
if $\lambda_1 = \lambda_2 \pm \lambda_3$.
Examples of functions having all of these properties include the families
\begin{alignat*}{2}
  \kappa_s &: B \to \R^+, &\qquad \kappa_s(z) &= \bigl( 1 -|z|^2 \bigr)^{-s} \\
  v_r &: \hat{B} \to \R^+, &  v_r(\lambda,\zeta) &= \bigl( 1 + |\lambda|^2 \bigr)^{r/2}
\end{alignat*}
which will be our main instruments for measuring (fast) decay and (slow) growth.
(Like $v_r$, the weight $\kappa_s$ could be called \emph{polynomial}; it becomes apparent that $\kappa_s$ corresponds to a polynomial weight when transferred to the hyperboloid model.)

Regarding growth, we shall need the following result concerning the asymptotics of the $\hc$-function.

\begin{prop}
  \label{thm:hc_poly}
  There exist constants $C_1, C_2$ such that
  \begin{equation*}
    \left| \hc (\lambda) \right|^{-1} \le C_1 + C_2 \left| \lambda \right|^{n-1/2}, \qquad \Re(i\lambda) \ge 0.
  \end{equation*}
\end{prop}

\begin{proof}
  Applying the duplication formula
  \begin{equation*}
    \Gamma (z) \Gamma (z+\tfrac{1}{2}) = 2^{1-2z} \sqrt{\pi}\, \Gamma(2z)
  \end{equation*}
  to $z = i\lambda/2$, we obtain
  \begin{equation*}
    \hc(\lambda)^{-1} = 2^{i\lambda-n} \frac{\Gamma \bigl(\frac{n+i\lambda}{2} \bigr)^2}{\Gamma(n) \Gamma(i\lambda)}
    = \frac{2^{1-n}}{\sqrt{\pi}\, \Gamma(n)} \frac{\Gamma \bigl(\frac{n+i\lambda}{2} \bigr)}{\Gamma \bigl(\frac{i\lambda}{2} \bigr)} \frac{\Gamma \bigl(\frac{n+i\lambda}{2} \bigr)}{\Gamma \bigl(\frac{1+i\lambda}{2} \bigr)}.
  \end{equation*}
  Now using
  \begin{equation*}
    \lim_{|z| \to \infty} \frac{\Gamma(z+\alpha)}{\Gamma(z)} z^{-\alpha} = 1, \qquad \left| \arg z \right| \le \pi - \delta
  \end{equation*}
  for every $\alpha \in \C$, $\delta > 0$ (see \cite[cor. 2.6.4]{beals_wong}), it is seen that the quotients on the right satisfy
  \begin{equation*}
    \lim_{|\lambda| \to \infty} \frac{\Gamma \big(\frac{n+i\lambda}{2} \big)}{\Gamma \big( \frac{i\lambda}{2})} \left( \frac{i\lambda}{2} \right)^{-{\frac{n}{2}}} = 1, 
    \quad \text{and} \quad
    \lim_{|\lambda| \to \infty} \frac{\Gamma \bigl( \frac{n+i\lambda}{2} \bigr)}{\Gamma \bigl( \frac{1+i\lambda}{2} \bigr)} \left( \frac{1+i\lambda}{2} \right)^{-{\frac{n-1}{2}}} = 1.
  \end{equation*}
  This proves $\left| \hc(\lambda) \right|^{-1} \in O \bigl(\left|\lambda \right|^{n-1/2} \bigr)$ as $\left| \lambda \right| \to \infty$.
\end{proof}

In other terms, we have
\begin{equation}
  \left| \hc(\lambda) \right|^{-2} \sim c_0\, v_{2n-1}(\lambda)
  \label{eq:hc_poly}
\end{equation}
asymptotically as $|\lambda| \to \infty$, for a suitable constant $c_0$.

Concerning fast decay, we note that by self-adjointness of the invariant Laplacian, smoothness translates directly to fast decay of the Helgason transform; if $f \in C^k(B) \cap \L^2(B,\mu)$, then $| \hat{f}(\lambda, \zeta) | \in O(v_k(\lambda))$ as $|\lambda| \to \infty$.

Let $\mathcal{S}(B)$ be the space of all smooth functions $f$ on $B$, for which all (semi-)norms
\begin{equation*}
  \left\Vert f \right\Vert_{(s)} = \sup_{|\alpha| \le s} \left\Vert \kappa_s \partial^\alpha f \right\Vert_{\L^\infty(\mu)} = \sup_{|\alpha| \le s} \left\Vert \kappa_{s+n+1} \partial^\alpha f \right\Vert_{\infty}
\end{equation*}
are finite.
Let $\mathcal{S}'(B)$ denote its dual, the space of continuous linear functionals on $\mathcal{S}(B)$, and $\mathcal{S}_\natural(B)$ the subspace of $\mathcal{S}(B)$ consisting of radial functions.
Then $\mathcal{S}(B)$ is invariant under the action of $\rho$, so that $\mathcal{S}_\natural(B)$ will provide a suitable reservoir of window functions.

\begin{lem}
  \hspace{0ex}
  \begin{enumerate}[(i)]
    \item For each $s > 0$ there exists $C_s > 0$ such that if $g \in \mathcal{S}(B)$ and $(w,b) \in B \times \hat{B}$, then $\rho(w,b) g \in \mathcal{S}(B)$ and
    \begin{equation*}
      \left\Vert \rho(w,b) g \right\Vert_{(s)} \le C_s \kappa_{1+s}(w) v_s (b) \left\Vert g \right\Vert_{(s')},
    \end{equation*}
    provided $s' \ge 3s + n + 1$.
  \item For $\psi \in \mathcal{S}(B)$, the mapping
    \begin{equation*}
      B \times \hat{B} \to \mathcal{S}(B), \quad (w,b) \mapsto \rho(w,b) \psi
    \end{equation*}
    is continuous.
  \end{enumerate}
  \label{lem:schwartz_inv}
\end{lem}

\begin{proof}
  \textit{(i)} Since
  \begin{equation*}
    P_{\lambda,\zeta} (z) = \biggl( \frac{1 - |z|^2}{\left| 1 - \left< z, \zeta \right> \right|^2} \biggr)^{\frac{n+i\lambda}{2}} \quad \text{and} \quad \varphi_w(z) = \frac{w + Q_w z}{1 - \left< z, w \right>}
  \end{equation*}
  are powers of rational functions in $z$ and $w$, applying a differential operator $\partial^\alpha$ in $z$ and $\bar{z}$ results in a sum
  \begin{equation}
    \partial^\alpha P_{\lambda,\zeta}(z) g(\varphi_w(z))
    = \sum_{\left| \beta \right| \le \left| \alpha \right|} R_\beta (z,w,\lambda,\zeta) (\partial^\alpha g) (\varphi_w(z)),
    \label{eq:sum_mod_rat}
  \end{equation}
  where the coefficients $R_\beta$ are again rational functions, and bounded in growth by
  \begin{equation*}
    \left| R_\beta (z,w,\lambda,\zeta) \right| \in 
    \begin{cases}
      O \big( \left|1 - \left< z, \zeta \right> \right|^{-n-|\beta|} \cdot \left| 1 - \left< z, w \right> \right|^{-1-|\beta|} \big), & |z| \to 1 \\
      O \big( | \lambda |^{|\beta|} \big), & |\lambda| \to \infty \\
      O \big( \left| 1 - \left< z, w \right>  \right|^{-1-|\beta|} \big), & |w| \to 1.
    \end{cases}
  \end{equation*}
  Using the estimate
  \begin{equation*}
    \left| 1 - \left< z, w \right> \right| \ge \frac{1}{2} \max \{  1 - |z|^2, 1 - |w|^2 \},
  \end{equation*}
  we see that these coefficients are bounded by
  \begin{equation*}
    \left| R_\beta(z,w,\lambda,\zeta) \right| \le C_\beta \kappa_{n+1+2|\beta|}(z) \kappa_{1+|\beta|} (w) v_{|\beta|}(\lambda).
  \end{equation*}
  Multiplying with $\kappa_{s}$ ($s \ge |\alpha|$) and extending this estimate to the whole sum in \eqref{eq:sum_mod_rat}, we obtain
  \begin{align*}
    \left\Vert \kappa_{s}\, \partial^\alpha \rho(w,b) g \right\Vert_\infty
    &\le C \kappa_{1+|\alpha|} (w) v_{|\alpha|} (b) \sup_{|\beta| \le |\alpha|} \bigl\Vert \kappa_{n+1+2|\alpha|+s}\, \partial^\beta g \bigr\Vert_\infty \\
    & \le C \kappa_{1+s} (w) v_{s} (b) \sup_{|\beta| \le s} \bigl\Vert \kappa_{n+1+3s}\, \partial^\beta g \bigr\Vert_\infty \nonumber \\
    & \le C \kappa_{1+s} (w) v_{s} (b) \bigl\Vert g \bigr\Vert_{(3s)}.
  \end{align*}
  Since the Norms $\left\Vert \cdot \right\Vert_{(s)}$ are monotone in $s$, the result follows.

  \textit{(ii)} In $B \times \hat{B}$, pick a convergent sequence $(w_j, b_j)_{j \in \N}$ with limit $(w, b)$.
  Since $\psi \in \mathcal{S}(B)$, every $\partial^\alpha \psi$ is uniformly continuous, and
  \begin{equation*}
    \partial^\alpha (\psi \circ \varphi_{w_j}) \longrightarrow \partial^\alpha (\psi \circ \varphi_w) \quad (j \to \infty)
  \end{equation*}
    converges uniformly.
    Next, let $G : B \times \hat{B} \times B \to \C$ be continuous and, for some $s' > 0$, satisfy
    \begin{equation}
      \left| G(w,b,z) \right| \in O \!\left(\kappa_{s'}(z) \right), \quad (w,b) \in B \times \hat{B}.
      \label{eq:G_small_growth}
    \end{equation}
     Let $\varepsilon > 0$.
     Then for every $f \in \mathcal{S}(B)$ there exists $r < 1$ so that
     \begin{equation*}
       \left| G(w_j, b_j, z) \cdot f(z) \right|,\ \left| G(w, b, z) \cdot f(z) \right| < \varepsilon
     \end{equation*}
     for all $z \in B \setminus \overline{rB}$.
     Furthermore, since $\{ (w, b), \ (w_j, b_j) : j \in \N \} \times \overline{rB}$ is compact, the restriction of $G$ to this set is uniformly continuous.
     Writing $g_j(z) = G(w_j, b_j, z)$ and $g(z) = G(w, b, z)$, this implies uniform convergence $g_j \to g$ on $\overline{rB}$.
     Let $f_j, f \in \mathcal{S}(B)$ ($j \in \N$) and $f_j \to f$ uniformly.
     We obtain the estimate
     \begin{align*}
       \sup_{z \in \overline{rB}} \left| g_j(z) f_j(z) - g(z) f(z) \right|
       &\le \sup_{z \in \overline{rB}} \left| g_j(z) \right| \left| f_j(z) - f(z) \right| + \left| f(z) \right| \left| g_j(z) - g(z) \right| \\
       &\le \sup_{z \in \overline{rB}} \left| g_j(z) \right| \left\Vert f_j - f \right\Vert_\infty + \left\Vert f \right\Vert_\infty \sup_{z \in \overline{rB}} \left| g_j(z) - g(z) \right|,
     \end{align*}
     and see that both terms can be made arbitrarily small by choosing $j$ large enough.
     Consequently, $\left\Vert G(w_j,b_j, \cdot) f_j - G(w, b, \cdot) f \right\Vert_\infty \to 0$ for $j \to \infty$.
     Since $\kappa_s \partial^\alpha \rho(w,b) \psi$ is always a finite sum of terms $G(w,b, \cdot) f$ of this form (for some $f \in \mathcal{S}(B)$, and with $G$ satisfying \eqref{eq:G_small_growth}), we have proved that $(w_j, b_j) \to (w, b)$ implies
     \begin{equation*}
       \left\Vert \rho(w_j,b_j) \psi - \rho(w, b) \psi \right\Vert_{(s)} \to 0
     \end{equation*}
     for arbitrary $s > 0$.
\end{proof}

If $\psi \in \mathcal{S}_\natural(B)$, the voice transform therefore extends to $\mathcal{S}'(B)$ by setting
\begin{equation*}
  V_\psi f(w,b) = \bigl< f, \rho(w,b) \psi \bigr>,
\end{equation*}
the term on the right denoting dual pairing.
Some statements carry over from the Euclidean STFT to the present case, resulting in the inversion formula on $\mathcal{S}'(B)$.

\begin{lem}
  \label{lem:dist_continuous}
  Let $\psi \in \mathcal{S}_\natural(B)$.
  If $f \in \mathcal{S}'(B)$, then $V_\psi f$ is continuous on $B \times \hat{B}$ and there exist constants $C, s, r > 0$ such that
  \begin{equation*}
    \left| V_\psi f(w,b) \right| \le C \kappa_s(w) v_r(b)
  \end{equation*}
  holds for all $w \in B,\, b \in \hat{B}$.
\end{lem}

\begin{proof}
  Since $\rho(w,b)\psi \in \mathcal{S}(B)$, continuity of $f \in \mathcal{S}'(B)$ implies
  \begin{equation*}
    \left| V_\psi f (w,b) \right| = \left| \left< f, \rho(w,b) \psi \right> \right| \le C_1 \left\Vert \rho(w,b) \psi \right\Vert_{(s_1)}
  \end{equation*}
for suitable $C_1, s_1 > 0$.
By lemma \ref{lem:schwartz_inv} \textit{(i)}, we may take $s, r$ sufficiently large such that the last norm is bounded up to a constant by $\kappa_s(w) v_r(b) \left\Vert \psi \right\Vert_{(s)}$.
Finally, continuity of $f$ and lemma \ref{lem:schwartz_inv} \textit{(ii)} also imply pointwise continuity of $V_\psi f$.
\end{proof}

\begin{lem}
  \label{lem:F_int_schwartz}
  Let $\psi \in \mathcal{S}_\natural(B)$ and let $F$ be measurable on $B \times \hat{B}$, satisfying
  \begin{equation}
    \left| F(w,b) \right| \le C_{s,r}\, \kappa_{-s}(w)\, v_{-r} (b)
    \label{eq:F_fast}
  \end{equation}
  for all $s, r > 0$, with $C_{s,r}$ independent of $(w,b)$.
  Then the integral
  \begin{equation*}
    z \mapsto f(z) = \iint_{B \times \hat{B}} F(w,b) \rho(w,b) \psi(z)\, \dd\mu(w) \dd\nu(b)
  \end{equation*}
  defines a function in $\mathcal{S}(B)$.
\end{lem}

\begin{proof}
  Since \eqref{eq:F_fast} is valid for all $s,r$, the integral converges absolutely, so it can be interchanged with any $\partial^\alpha$, yielding the estimate
  \begin{align}
    \left\Vert \kappa_s \partial^\alpha f \right\Vert_{\infty}
    &\le \iint \left| F(w,b) \right| \left\Vert \kappa_s \partial^\alpha \rho(w,b) \psi \right\Vert_{\infty} \dd\mu(w) \dd\nu(b) \nonumber \\
    &\le C \iint \left| F(w,b) \right| \kappa_{s'}(w) v_{r'}(b) \, \dd\mu(w) \dd\nu(b),
    \label{eq:schwartz_est_style}
  \end{align}
  This last term is finite, again by lemma \ref{lem:schwartz_inv}
\end{proof}

\begin{lem}
  \label{lem:V_schwartz}
  Let $\psi \in \mathcal{S}_\natural(B)$ and $f \in \mathcal{S}'(B)$.
  Then the following are equivalent:
  \begin{enumerate}[(i)]
    \item $f \in \mathcal{S}(B)$.
    \item For all $s, r > 0$ there exists $C_{s,r} > 0$ so that
      \begin{equation*}
	\left| V_\psi f (w, b) \right| \le C_{s,r} \, \kappa_{-s}(w) v_{-r}(b),
        \qquad \text{$w \in B$, $b \in \hat{B}$.}
      \end{equation*}
  \end{enumerate}
\end{lem}

\begin{proof}
  Let $f \in \mathcal{S}(B)$ and let $b \in \hat{B}$ be fixed.
  Then $w \mapsto V_\psi f (w,b)$ is the convolution (cf. \eqref{eq:convolution_B}) of two functions in $\mathcal{S}(B)$, so $\kappa_s V_\psi f (\cdot, b)$ is bounded for arbitrary $s > 0$.
  On the other hand, $b \mapsto V_\psi f (w,b)$ is the Helgason transform of a smooth function, and possesses fast decay.
  Hence $v_r V_\psi f(w, \cdot)$ is bounded for every $r > 0$.

  The other implication follows from lemma \ref{lem:F_int_schwartz}; since the integral
  \begin{equation*}
    \tilde{f} = \left< \psi, \psi \right>^{-1} \iint V_\psi f (w,b) \rho(w,b) \psi \, \dd\mu(w) \dd\nu(b)
  \end{equation*}
  defines a function in $\mathcal{S}(B)$, the inversion formula on $\L^2(B,\mu)$ yields $\tilde{f} = f$.
\end{proof}

\begin{lem}
  \label{lem:V_equiv_norms}
  Let $\psi \in \mathcal{S}_\natural(B)$.
  Then the seminorms $\left\Vert \cdot \right\Vert_{s,r}$ defined by
  \begin{equation*}
    \left\Vert f \right\Vert_{s,r} = \sup_{(w,b) \in B \times \hat{B}} \kappa_s(w) v_r(b) \left| V_\psi f (w, b) \right|
  \end{equation*}
  generate the same topology on $\mathcal{S}(B)$ as $\left\Vert \cdot \right\Vert_{(s)}$.
\end{lem}

\begin{proof}
  Set $\tilde{\mathcal{S}}(B) = \{ f \in \L^2(B,\mu) : \left\Vert f \right\Vert_{s,r} < \infty \text{ for all } s, r > 0 \}$.
  Then, by lemma \ref{lem:V_schwartz}, $f \in \mathcal{S}(B)$ if and only if $f \in \tilde{\mathcal{S}}(B)$.
  A similar estimate as in \eqref{eq:schwartz_est_style} yields
  \begin{align*}
    \left\Vert \kappa_s \partial^\alpha f \right\Vert_{\infty}
    &\le C_1 \iint \left| V_\psi f (w,b) \right| \kappa_{s'} (w) v_{r'} (b) \, \dd\mu(w) \dd\nu(b) \\
    &\le C_2 \left\Vert f \right\Vert_{s'+n+1,\, r'+2n+1} \int_B \kappa_{-n-1}\, \dd\mu \int_{\hat{B}} v_{-2n-1}\, \dd\nu \\
    &\le C_3 \left\Vert f \right\Vert_{s'+n+1,\, r'+2n+1},
  \end{align*}
  for $s', r'$ large enough.
  This shows that the identity $\id : \tilde{\mathcal{S}}(B) \to \mathcal{S}(B)$ is continuous.
  By the open mapping theorem, this is also true for its inverse.
\end{proof}

After these preparations we are ready to state the inversion Formula on $\mathcal{S}'(B)$.

\begin{thm}
  \label{thm:dist_inv}
  Let $\gamma, \psi \in \mathcal{S}_\natural(B)$ and $\left< \gamma, \psi \right> \ne 0$.
  \begin{enumerate}[(i)]
    \item Suppose there exist constants $s,r, C_{s,r} > 0$ so that for all $w \in B$, $b \in \hat{B}$
      \begin{equation}
        \left| F(w,b) \right| \le C_{s,r}\, \kappa_s(w) v_r (b).
        \label{eq:F_growth}
      \end{equation}
      Then the integral $f = \iint F(w,b) \rho(w,b) \gamma \, \dd\mu(w) \dd\nu(b)$ defines an element in $\mathcal{S}'(B)$ via
      \begin{equation}
        \left< f, g \right> = \iint F(w,b) \left< \rho(w,b) \gamma, g \right> \dd\mu(w) \dd\nu(b).
        \label{eq:V_inverse_dist}
      \end{equation}
    \item In particular, for every $f \in \mathcal{S}'(B)$ we have the inversion formula
      \begin{equation}
        f = \left< \gamma, \psi \right>^{-1} \iint V_\psi f(w, b) \rho(w, b) \gamma \, \dd\mu(w) \dd\nu(b).
        \label{eq:V_inverse}
      \end{equation}
  \end{enumerate}
\end{thm}

\begin{proof}
  \textit{(i)} Let $g \in \mathcal{S}(B)$, then by lemma \ref{lem:V_schwartz}, the term
  \begin{equation*}
    \kappa_{s'}(w) v_{r'}(b) \left< \rho(w,b) \gamma, g \right>
    = \kappa_{s'}(w) v_{r'}(b) \overline{V_\gamma g (w,b)}
  \end{equation*}
    is bounded for all $s', r'$.
  Choosing these large enough, the integral in \eqref{eq:V_inverse_dist} is thus absolutely convergent and, like above, we can estimate
  \begin{align*}
    \left| \left< f, g \right> \right|
    &\le \iint \left| F(w,b) \right| \left| V_\gamma g(w,b) \right| \dd\mu(w) \dd\nu(b) \\
    &\le C \sup_{(w,b) \in B \times \hat{B}} \left| \kappa_{s+s'+n+1} (w) v_{r+r'+2n+1}(b) \right|.
  \end{align*}
  By lemma \ref{lem:V_equiv_norms}, this shows $f \in \mathcal{S}'(B)$.

  \textit{(ii)} Conversely, let $f \in \mathcal{S}'(B)$.
  By lemma \ref{lem:dist_continuous}, $F = V_\psi f$ is continuous and satisfies the growth condition \eqref{eq:F_growth}.
  The integral
  \begin{equation*}
    \bigl< \tilde{f}, g \bigr> = \left< \gamma, \psi \right>^{-1} \iint V_\psi f (w,b) \left< \rho(w,b) \gamma, g \right> \dd\mu(w) \dd\nu(b)
  \end{equation*}
  thus defines an element $\tilde{f} \in \mathcal{S}'(B)$.
  Recalling the inversion formula for $g \in \mathcal{S}(B) \subset \L^2(\mu)$, now with window $\gamma$,
  \begin{equation*}
    g = \left< \psi, \gamma \right>^{-1} \iint V_\gamma g(w,b) \rho(w,b) \psi\, \dd\mu(w) \dd\nu(b),
  \end{equation*}
  we see that
  \begin{equation*}
    \bigl< \tilde{f}, g \bigr> = \left< \gamma, \psi \right>^{-1} \iint \left< f, \rho(w,b) \psi \right> \overline{V_\gamma g(w,b)}\, \dd\mu(w) \dd\nu(b) = \left< f, g \right>.
  \end{equation*}
  This implies $\tilde{f} = f$.
\end{proof}

\subsection{Coorbit Spaces}

A more refined analysis of the voice transform is made possible by considering certain subsets of $\mathcal{S}'(B)$, namely those distributions whose transforms feature decay rates measured by fixed classes of weight functions.
Specifically, let $1 \le p \le \infty$ and let $m: B \times \hat{B} \to \R^+$ be a weight function.
The weighted Lebesgue space
\begin{equation*}
  \L_m^p = \L_m^p(\mu{\otimes}\nu)
\end{equation*}
is of particular interest, consisting of all measurable $F : B \times \hat{B} \to \C$, for which the norm
\begin{equation*}
  \left\Vert F \right\Vert_{\L_m^p} = \left\Vert F \cdot m \right\Vert_{\L^p(\mu \otimes \nu)}
\end{equation*}
is finite.
Its dual is $\L_{1/m}^q$, with $\frac{1}{p} + \frac{1}{q} = 1$.

It is convenient to assume $m$ to be of the form $m(z,b) = \kappa(z) v(b)$, with $\kappa$ and $v$ both radial and submultiplicative.

\bigskip

Let $\psi \in \mathcal{S}_\natural(B)$ be a fixed window.
We define the (weighted) \emph{coorbit space} $M_m^p$ to be the set
\begin{equation*}
  M_m^p = \left\{ f \in \mathcal{S}'(B) : V_\psi f \in \L_m^p \right\}.
\end{equation*}
Also, define the formal adjoint operator $V^*_\psi$ of $V_\psi$ by the integral
\begin{equation*}
  V^*_\psi F = \iint F(w,b) \rho(w,b) \psi \, \dd\mu(w)\dd\nu(b),
\end{equation*}
which is understood in the weak sense, satisfying the identity
\begin{align*}
  \bigl< V^*_\psi F, g \bigr> &= \iint F(w,b) \left< \rho(w,b) \psi , g \right> \dd\mu(w)\dd\nu(b) \\
  &= \iint F(w,b) \overline{V_\psi g (w,b)} \, \dd\mu(w)\dd\nu(b) \\
  &= \left< F, V_\psi g \right>.
\end{align*}
We have the following basic result about coorbit spaces.

\begin{thm}
  \label{thm:coorb_invers}
  Let $m$ satisfy $m(z,b) \le \kappa_s(z) v_r(b)$ for some $s, r > 0$, and let $M_m^p$ be defined in terms of a fixed window $\psi \in \mathcal{S}_\natural(B)$.
  Let $\gamma \in \mathcal{S}_\natural(B)$ satisfy $\left< \psi, \gamma \right> \ne 0$.
  Then the following holds:
  \begin{enumerate}[(i)]
    \item The adjoint operator $V^*_\gamma : \L_m^p \to M_m^p$ is continuous.
    \item The inversion formula \eqref{eq:V_inverse} holds for every $f \in M_m^p$, in other terms,
      \begin{equation*}
        f = \left< \gamma, \psi \right>^{-1} V^*_\gamma V_\psi f.
      \end{equation*}
    \item The coorbit space arising from $\gamma$ is the same as $M_m^p$, with equivalent norms.
  \end{enumerate}
\end{thm}

\begin{proof}
  \textit{(i)} If $\frac{1}{p} + \frac{1}{q} = 1$, it follows from Hölder's inequality that for every $g \in \mathcal{S}(B)$,
  \begin{equation*}
    \bigl| \bigl< V^*_\gamma F, g \bigr> \bigr| = \bigl| \bigl< F, V_\gamma g \bigr> \bigr|
    \le \left\Vert F \right\Vert_{\L_m^p} \left\Vert V_\gamma g \right\Vert_{\L_{1/m}^q}
    \le \left\Vert F \right\Vert_{\L_m^p} \left\Vert g \right\Vert_{s,r} \left\Vert \kappa_{-s} v_{-r} \right\Vert_{\L_{1/m}^q}.
  \end{equation*}
  The last term is finite for $s,r > 0$ large enough, which shows continuity of $V^*_\gamma F$ with regard to the equivalent seminorms $\left\Vert \cdot \right\Vert_{s,r}$ on $\mathcal{S}(B)$, i.e.\ $V^*_\gamma F \in \mathcal{S}'(B)$.
  Its voice transform $V_\psi V^*_\gamma F$ is therefore continuous, and we have the following pointwise estimate,
  \begin{align*}
    \bigl| V_\psi V^*_\gamma F (w,b) \bigr|
    &= \bigl| \bigl< V^*_\gamma F, \rho(w,b) \psi \bigr> \bigr|
    = \bigl| \bigl< F, V_\gamma \bigl( \rho(w,b) \psi \bigr) \bigr> \bigr| \\
    &\le \iint \left| F(w',b') \right| \left| V_\gamma \bigl( \rho(w,b) \psi \bigr) (w',b') \right| \dd\mu(w') \dd\nu(b').
  \end{align*}
  Now by lemma \ref{lem:V_schwartz}, $V_\gamma f$ has fast decay whenever $f \in \mathcal{S}(B)$.
  Since
  \begin{equation*}
    V_\gamma \bigl( \rho(w,b) \psi \bigr) (w',b') = \left< \rho(w,b)\psi, \rho(w',b')\gamma \right> = \overline{V_\psi \bigl( \rho(w',b') \gamma \bigr) (w,b)},
  \end{equation*}
  this term has fast decay in all variables.
  Thus, by the weighted Young inequality (theorem \ref{thm:young}), we obtain
  \begin{equation}
    \label{eq:young_coorb1}
    \bigl\Vert V^*_\gamma F \bigr\Vert_{M_m^p} = \bigl\Vert V_\psi V^*_\gamma F \bigr\Vert_{\L_m^p} \le C \left\Vert F \right\Vert_{\L_m^p}.
  \end{equation}

  \textit{(ii)} is now immediate: If $f \in M_m^p$, then $V_\psi f \in \L_m^p$, and $\tilde{f} = \left< \gamma, \psi \right>^{-1} V^*_\gamma V_\psi f$ defines an element in $M_m^p$.
  Since $M_m^p \subset \mathcal{S}'(B)$, theorem \ref{thm:dist_inv} implies $\tilde{f} = f$.

  \textit{(iii)} By the above,
  \begin{equation*}
    \left\Vert V_\psi f \right\Vert_{\L_m^p} = \left\Vert f \right\Vert_{M_m^p} = \left< \gamma, \gamma \right>^{-1} \bigl\Vert V^*_\gamma V_\gamma f \bigr\Vert_{M_m^p} \le \tilde{C}_\gamma \left\Vert V_\gamma f \right\Vert_{\L_m^p}
  \end{equation*}
  holds for every nonzero $\gamma \in \mathcal{S}_\natural(B)$.
  Interchanging $\psi$ and $\gamma$ we obtain
  \begin{equation*}
    \left\Vert V_\gamma f \right\Vert_{\L_m^p} \le \tilde{C}_\psi \left\Vert V_\psi f \right\Vert_{\L_m^p},
  \end{equation*}
  so both norms are equivalent.
\end{proof}

Coorbit spaces are Banach spaces, and they contain $\mathcal{S}(B)$ as a dense subspace.

\begin{thm}
  Let $r, s > 0$ be fixed and let $m(z,b) \le \kappa_s (z) v_r (b)$.
  If $1 \le p < \infty$, then $\mathcal{S}(B)$ is a dense subspace of $M_m^p$.
\end{thm}

\begin{proof}
  The inclusion $\mathcal{S}(B) \subset M_m^p$ follows from
  \begin{equation*}
    \left\Vert f \right\Vert_{M_m^p} = \left\Vert V_\psi f \right\Vert_{\L_m^p}
    \le \left\Vert \kappa_s v_r V_\psi f \right\Vert_{\infty} \left\Vert \kappa_{-s} v_{-r} \right\Vert_{\L_m^p},
  \end{equation*}
  which is finite, if $f \in \mathcal{S}(B)$ and $s, r > 0$ are chosen big enough.

  Pick a real sequence $(r_j)$ featuring $0 < r_j < 1$ and $r_j \nearrow 1$.
  Then the sets
  \begin{equation*}
    K_j = \{ (z, \lambda) \in B \times \R : |z| \le r_j, |\lambda| \le j \} \times S
  \end{equation*}
  form an exhausting sequence in $B \times \hat{B}$.
  Let $f \in M_m^p$, set $F_j (z,b) = V_\psi f (z,b) \mathbf{1}_{K_j} (z,b)$ ($\mathbf{1}_{K}$ denotes the characteristic function on $K$), and $f_j = V_\psi^* F_j$.
  Then every $F_j$ has fast decay, thus $f_j \in \mathcal{S}(B)$ by lemma \ref{lem:V_schwartz}.
  We may assume $\left< \psi, \psi \right> = 1$, then by theorem \ref{thm:coorb_invers},
  \begin{equation*}
    \left\Vert f - f_j \right\Vert_{M_m^p}
    = \left\Vert V_\psi^* V_\psi (f - f_j) \right\Vert_{M_m^p}
    = \left\Vert V_\psi^* (V_\psi f - F_j) \right\Vert_{M_m^p}
    \le C \left\Vert V_\psi f - F_j \right\Vert_{\L_m^p}.
  \end{equation*}
  If $p < \infty$, then $\left\Vert V_\psi f - F_j \right\Vert_{\L_m^p} \to 0$, therefore $f_j \to f$ in $M_m^p$.
\end{proof}

\begin{thm}
  Impose the same assumption on $m$ as before.
  Then $M_m^p$ is a Banach space for every $1 \le p \le \infty$.
  Its dual is $(M_m^p)' = M_{1/m}^q$, where $1/p + 1/q = 1$.
\end{thm}

\begin{proof}
  Let $\psi \in \mathcal{S}_\natural(B)$ with $\left\Vert \psi \right\Vert_{\L^2(\mu)} = 1$.
  Then $VM_m^p$ is a subspace of $\L_m^p$ and isometrically isomorphic to $M_m^p$.
  Let $(f_j)$ be a Cauchy sequence in $M_m^p$, then $(F_j)$, with $F_j = V_\psi f_j$, is a Cauchy sequence in $\L_m^p$, and converges to a unique element $F \in \L_m^p$.
  We set $f = V_\psi^*F$, then again by theorem \ref{thm:coorb_invers},
  \begin{equation*}
    \left\Vert f_j - f \right\Vert_{M_m^p}
    = \left\Vert V_\psi^* (V_\psi f_j - F) \right\Vert_{M_m^p}
    \le C \left\Vert V_\psi f_j - F \right\Vert_{\L_m^p}
  \end{equation*}
  This shows $f_j \to f$ in $M_m^p$ and $V_\psi f = F$.
  Thus $VM_m^p$ is closed, and $M_m^p$ is complete.

  The proof of the second statement is virtually identical to theorem 11.3.6 in \cite{groechenig}.
\end{proof}

\subsection{Reproducing Kernel}

Let $\psi \in \mathcal{S}_\natural(B)$ be a fixed window with $\left< \psi, \psi \right> = 1$.
Then the inversion formula on $M_m^p$ takes the simpler form $V^*_\psi V_\psi f = f$, and by applying $V_\psi$, this can be rewritten as
\begin{equation*}
  V_\psi f(X) = \bigl< V^*_\psi V_\psi f, \rho(X) \psi \bigr> = \big< V_\psi f, V_\psi \bigl( \rho(X)\psi \bigr) \bigr>
  = \big< V_\psi f, R(X, \cdot) \bigr>
\end{equation*}
where
\begin{equation*}
  R(X,Y) = V_\psi \bigl( \rho(X) \psi \bigr) (Y) = \left< \rho(X)\psi, \rho(Y)\psi \right> = \overline{R(Y, X)}
\end{equation*}
for $X, Y \in B \times \hat{B}$.
Thus, $R$ serves as a reproducing kernel on the space
\begin{equation*}
  VM_m^p = \left\{ V_\psi f : f \in M_m^p \right\}.
\end{equation*}
We observe that the prerequisites of theorem \ref{thm:coorb_invers} can be stated in terms of $R$, since they lead to the following integrability condition, which suffices to apply the weighted Young inequality in \eqref{eq:young_coorb1}.

\begin{lem}
  \label{lem:kernel_int}
  Let $\psi \in \mathcal{S}_\natural(B)$ and $m(w,b) \le \kappa_s(w) v_r(b)$. Then the integrals
  \begin{equation*}
    \iint\limits_{B \times \hat{B}} \bigl| R(X,Y) \bigr| \frac{m(X)}{m(Y)} \, \dd(\mu{\otimes}\nu)(Y) \quad \text{and} \quad  \iint\limits_{B \times \hat{B}} \bigl| R(X,Y) \bigr| \frac{m(X)}{m(Y)} \, \dd(\mu{\otimes}\nu)(X) \qquad 
  \end{equation*}
  are bounded by a constant $C_\psi$, independent of $X$ and $Y$ respectively.
\end{lem}

So far, we have only considered weights of polynomial growth.
The extension to general weights will require to impose additional restrictions on the kernel $R$.
To simplify notation, we write points in $B \times \hat{B}$ in capital letters and denote by
\begin{equation*}
  \xi = \mu \otimes \nu
\end{equation*}
the product measure on this space.

We call $\psi \in \mathcal{S}_\natural(B)$ \emph{admissible}, if there exists a constant $C_\psi$ so that for all $X \in B \times \hat{B}$,
\begin{equation}
  \iint \left| R(X, Y) \right| \frac{m(Y)}{m(X)}\, \dd\xi(Y) \le C_\psi.
  \label{eq:kernel_int_cond}
\end{equation}
Note that this implies $R(X, \cdot) \in \L_m^1$.
Let $\psi$ be admissible, set
\begin{equation*}
  H_m^1 = \{ f \in \L^2(\mu) : V_\psi f \in \L_m^1 \},
\end{equation*}
and let $(H_m^1)'$ denote the space of continuous linear functionals on $H_m^1$.
By this duality, $V_\psi$ extends to $(H_m^1)'$ by setting
\begin{equation*}
  V_\psi f (X) = \left< f, \rho(X) \psi \right>,
\end{equation*}
which is well defined, since $\rho(X) \psi \in H_m^1$ by \eqref{eq:kernel_int_cond}.
By imposing a minimal growth condition on $m$, we may further assume that
\begin{equation*}
  \sup_{X \in B \times \hat{B}} \left| R(X, X) \right| m(X)^{-1} \le C_\psi'.
\end{equation*}
We then obtain the continuous and dense embeddings
\begin{equation}
  H_m^1 \hookrightarrow \L^2(B,\mu) \hookrightarrow (H_m^1)'.
  \label{eq:dense_embed}
\end{equation}
Indeed, let us consider the first inclusion.
Since $H_m^1 \subset \L^2(\mu)$, we may use \eqref{eq:orth_rel} and the Cauchy-Schwarz inequality to obtain
\begin{align*}
  \left\Vert f \right\Vert_{\L^2(\mu)}^2 
  = \left\Vert V_\psi f \right\Vert_{\L^2(\mu)}^2
  &= \iint \left| \left< f, \rho(X) \psi \right> \right| \left| V_\psi f(X) \right| \dd\xi(X) \\
  &= \iint \left| \left< f, m(X)^{-1} \rho(X) \psi \right> \right| \left| V_\psi f(X) \right| m(X)\, \dd\xi(X) \\
  &\le \left\Vert f \right\Vert_{\L^2(\mu)} \iint \left\Vert m(X)^{-1} \rho(X) \psi \right\Vert_{\L^2(\mu)} \left| V_\psi f(X) \right| m(X)\, \dd\xi(X) \\
  &= \left\Vert f \right\Vert_{\L^2(\mu)} \iint \left| m(X)^{-1} R(X,X) \right| \left| V_\psi f(X) \right| m(X)\, \dd\xi(X) \\
  &\le C_\psi' \left\Vert f \right\Vert_{\L^2(\mu)} \left\Vert V_\psi f \right\Vert_{\L_m^1}.
\end{align*}
Therefore, $\left\Vert f \right\Vert_{\L^2(\mu)} \le C_\psi' \left\Vert f \right\Vert_{H_m^1}$ and the first inclusion is continuous.
Next, $V_\psi$ is injective on $\L^2(\mu)$, which means the set
\begin{equation*}
  \{ \rho(X) \psi : X \in B \times \Hat{B} \}
\end{equation*}
is total in $H_m^1$.
Hence, the first embedding is dense.

The second embedding is obtained by a dual result; we refer to lemma A.1 in \cite{da_st_te2}.

\begin{thm}
  \hspace{0ex}
  \begin{enumerate}[(i)]
    \item The operator $V_\psi : (H_m^1)' \to \L_{1/m}^\infty$ is bounded and injective, furthermore, $V_\psi f$ is continuous for every $f \in (H_m^1)'$.
    \item The operator $V_\psi^* : \L_{1/m}^\infty \to (H_m^1)'$ defined by
      \begin{equation*}
	\left< V_\psi^* F, g \right> = \left< F, V_\psi g \right>
      \end{equation*}
      is bounded.
  \end{enumerate}
\end{thm}

\begin{proof}
  \textit{(i)} Let $f \in (H_m^1)'$ with operator norm $\left\Vert f \right\Vert_{(H_m^1)'}$.
  We calculate
  \begin{align*}
    \left\Vert V_\psi f \right\Vert_{\L_{1/m}^\infty} &= \sup_X \left| \left< f, \rho(X) \psi \right> \right| m(X)^{-1} \\
    &\le \left\Vert f \right\Vert_{(H_m^1)'} \sup_X \left. R (X, X\right) m(X)^{-1} \\
    &\le C_\psi' \left\Vert f \right\Vert_{(H_m^1)'}.
  \end{align*}
  Furthermore, since the set of functions $\rho(X) \psi$ is total in $\L^2(\mu)$, injectivity of $V_\psi$ extends to $(H_m^1)'$ by \eqref{eq:dense_embed}.

  \textit{(ii)} For $F \in \L_{1/m}^\infty$ we have
  \begin{align*}
    \left\Vert V_\psi^* F \right\Vert_{(H_m^1)'} &= \sup_{\left\Vert g \right\Vert_{H_m^1} = 1} \left| \left< V_\psi^* F, g \right> \right| \\
    &= \sup_{\left\Vert g \right\Vert_{H_m^1} = 1} \left| \left< F, V_\psi g \right> \right| \\
    &\le \sup_{\left\Vert g \right\Vert_{H_m^1} = 1} \left\Vert F \right\Vert_{\L_{1/m}^\infty} \left\Vert V_\psi g \right\Vert_{\L_m^1}
    = \left\Vert F \right\Vert_{\L_{1/m}^\infty}.
  \end{align*}
\end{proof}

We may now define coorbit spaces just like before, replacing our reservoir of distributions, i.e.
\begin{equation*}
  M_m^p = \{ f \in (H_m^1)' : V_\psi f \in \L_m^p \},
\end{equation*}
with norm $\left\Vert f \right\Vert_{M_m^p} = \left\Vert V_\psi f \right\Vert_{\L_m^p}$, and
\begin{equation*}
  VM_m^p = \{ V_\psi f : f \in M_m^p \}.
\end{equation*}

\begin{thm}
  Let $\psi$ be an admissible window, then the following holds.
  \begin{enumerate}[(i)]
    \item For every $f \in M_m^p$, $X \in B\times\hat{B}$,
      \begin{equation*}
	V_\psi f (X) = \left< V_\psi f , R(X, \cdot) \right>.
      \end{equation*}
    \item The operator $V_\psi^* V_\psi$ is the identity on $M_m^p$.
    \item $M_m^p$ is independent of $\psi$; different admissible windows give equivalent norms.
  \end{enumerate}
\end{thm}

\begin{proof}
  Both sides each define a continuous linear operator on $(H_m^1)'$, and both coincide on $\L^2(\mu)$.
  By the dense embedding $\L^2(\mu) \hookrightarrow (H_m^1)'$, they also coincide on $(H_m^1)'$.

  For the second assertion we calculate for $h \in H_m^1$ and $F \in \L_{1/m}^\infty$,
  \begin{align*}
    \left< V_\psi^* F, h \right> &= \left< F, V_\psi h \right>
    = \iint F(X) \left< \rho(X) \psi, h \right> \dd\xi(X) \\
    &= \Bigl< \iint F(X) \rho(X) \psi \, \dd\xi(X), h \Bigr>,
  \end{align*}
  so $V_\psi^* F = \iint F(X) \rho(X) \psi\, \dd\xi(X)$.
  Moreover,
  \begin{align*}
    V_\psi (V_\psi^* F) (X) = \left< V_\psi^* F, \rho(X) \psi \right>
    = \left< F, V_\psi \bigl( \rho(X) \psi \bigr) \right>
    = \left< F, R(X, \cdot) \right>.
  \end{align*}
  Setting $F = V_\psi f$, the above equation and the injectivity of $V_\psi$ give $V_\psi^* V_\psi f = f$.

  The third assertion is proved in a similar fashion as theorem \ref{thm:coorb_invers} \textit{(iii)}.
\end{proof}

\subsection{Translation Invariance}

A key feature of Euclidean modulation spaces is their invariance under time-frequency shifts (and, under certain conditions on the weight, under the Fourier transform).
More general, in the setting of group representations, coorbit spaces are defined in terms of a Banach function space, still required to be invariant under group action, see \cite{fei_groe1}.

While the question of such a translation invariance is motivated mainly in the group theoretic setting, it may also arise in the present case.
Specifically, consider a generic point $X = (w,b) \in B \times \hat{B}$. 
Then $X$ induces a translation $\tau_X$,
\begin{equation*}
  \tau_X (z, a) = (\varphi_w(z), t_b(a)),
\end{equation*}
where $t_b$ consists of a scalar translation
\begin{equation*}
  a = (x, \zeta) \mapsto (y - x, \zeta), \quad b = (y,\zeta')
\end{equation*}
and a possible unitary map in the second variable (which we omit).

\begin{lem}
  The spaces $VM_m^p$ are invariant under the translations $\tau_X$ ($X \in B \times \hat{B}$).
\end{lem}

\begin{proof}
  Choose $c_0 > 0$ such that
  \begin{equation*}
    \left| \hc(\lambda) \right|^{-2} \le c_0\, v_{2n-1}(\lambda), \quad \lambda \in \R
  \end{equation*}
  and set $\dd\tilde{\nu}(\lambda,\zeta) = v_{2n-1}(\lambda) \dd\lambda\, \dd\sigma(\zeta)$.
  Let $f \in M_m^p$, we then obtain for $X = (w,b)$,
  \begin{align*}
    \left\Vert (V_\psi f) \circ \tau_X \right\Vert_{\L_m^p}
    &= \left\Vert (V_\psi f) \circ \tau_X \cdot m \right\Vert_{\L^p(\mu\otimes\nu)} \\
    &\le c_0 \left\Vert (V_\psi f) \circ \tau_X \cdot m \right\Vert_{\L^p(\mu\otimes\tilde{\nu})} \\
    &\le c_0\, v_{2n-1}(b)\, m(X) \left\Vert V_\psi f \cdot m \right\Vert_{\L^p(\mu\otimes\tilde{\nu})}
  \end{align*}
  by invariance of $\mu$ and submultiplicativity of $m$ and $v_{2n-1}$.
  Since $V_\psi f$ is continuous, the last norm is finite.
\end{proof}

Note however, that the weighted spaces $\L_m^p$ are \emph{not} invariant under these translations due to the zero $\left| \hc(0) \right|^{-2} = 0$, which allows for singularities (up to order two) along the hypersurface $\lambda = 0$.
From this point of view, it seems more practical to replace
\begin{equation*}
  \dd\nu(\lambda,\zeta) = \left| \hc(\lambda) \right|^{-2} \dd\sigma(\zeta)\, \dd\lambda \quad \text{by} \quad \dd\tilde{\nu}(\lambda,\zeta) = v_{2n-1}(\lambda)\, \dd\sigma(\zeta)\, \dd\lambda
\end{equation*}
(omitting constants), and to define coorbit spaces in terms of $\L_m^p(\mu \otimes \tilde{\nu})$ instead of $\L_m^p(\mu \otimes \nu)$, i.e.
\begin{equation*}
  \left\Vert f \right\Vert_{\tilde{M}_m^p} = \left\Vert V_\psi f \cdot m \right\Vert_{\L^p(\mu \otimes \tilde{\nu})}.
\end{equation*}
By submultiplicativity of $v_{2n-1}$, the spaces $\L_m^p(\mu\otimes\tilde{\nu})$ are translation invariant, and $\tilde{M}_m^p \subset M_m^p$ holds, since $\L^p(\mu \otimes \tilde{\nu}) \subset \L^p(\mu \otimes \nu)$.
Furthermore, it is easily seen that all theorems in the preceding section still hold for $\tilde{M}_m^p$, in particular the inversion formula and the Banach space property.

\section{Frame theory}
\label{sct:frames}

Having established the continuous setting, we make extensive use of the reproducing kernel and derive atomic decompositions on the spaces $VM_m^p$.
By inversion of the voice transform, these decompositions give rise to Banach frames on the coorbit spaces $M_m^p$.

This section is mainly inspired by the work of Dahlke, Steidl, and Teschke, \cite{da_st_te1}, \cite{da_st_te2}.

\subsection{Partitions of Unity}

The choice of suitable partitions of unity on phase space $B \times \hat{B}$ is a starting point for the discretization.
The procedure for voice/wavelet transforms on a homogeneous space, arising from group representations, is to establish partitions of unity on group level first (see also \cite{fei_char_hom}), then transport them to the homogeneous space through composition with a section.
While this works in the case of the ball, the result may be obtained directly, for example via the following construction.

Let $U_o \subset B$ be a relatively compact neighborhood of the origin with non-empty interior. Then there exists a countable collection $(z_j)_{j \in J}$ of points in $B$, which is \emph{well spread} in the following sense.
\begin{enumerate}
  \item The sets $U_j = \varphi_{z_j}(U_o)$ cover $B$.
  \item There exists a partition $J = J_1 \cup \dots \cup J_{r_0}$ so that $U_i \cap U_j = \varnothing$ whenever $i,j$ are in the same index set $J_r$.
\end{enumerate}
Thus, $B$ decomposes as
\begin{equation}
  B = \bigcup_{j \in J} U_j = \bigcup_{r = 1}^{r_0} \biguplus_{j \in J_r} U_j,
  \label{eq:cov_B}
\end{equation}
where $\biguplus$ denotes disjoint union.
By construction, $\mu(U_j) = \mu(U_o)$ holds for all $j \in J$.

In the case of $\hat{B}$, we aim for a covering $(V_k)_{k \in K}$ with similar properties.
First, since $S$ is a finite union of $\sigma$-invariant subsets, it remains only to cover the real line.
Let $I_0 \subset \R$ be a relatively compact interval centered around the origin, and let $I_y = t_y^v(I_0)$ for $y \in \R$ denote its image under the weighted translation
\begin{equation*}
  t_y^v: \R \to \R, \quad x \mapsto y + \frac{x}{v_{2n-1}(y)}.
\end{equation*}
Using \eqref{eq:hc_poly}, it is easily verified that there exist positive constants $c,d$ such that
\begin{equation*}
  c \le \nu(I_y \times S) \le d
\end{equation*}
holds uniformly in $y$.
A cover
\begin{equation*}
  \R = \bigcup_{k \in K} t_{x_k}^v(I_0) = \bigcup_{s = 1}^{s_0} \biguplus_{k \in K_s} t_{x_k}^v(I_0)
\end{equation*}
can now be obtained from a suitable choice of a countable collection $(x_k)_{k \in K}$.
Extending this to $S$ (either trivially, or) through a covering consisting of $\sigma$-invariant subsets $S_k$, we obtain
\begin{equation}
  \hat{B} = \R \times S = \bigcup_{k \in K} V_k = \bigcup_{s = 1}^{s_0} \biguplus_{k \in K_s} V_k
  \label{eq:cov_R}
\end{equation}
where $V_k = t_{x_k}^v(I_0) \times S_k$.
Setting $x_0 = 0$, we have $V_0 = I_0 \times S_0$, and by the above construction, $\nu(V_k)$ is uniformly bounded from above and from below by a constant multiple of $\nu(V_0)$.

\bigskip

Given a covering of $B \times \hat{B}$ with sets $\mathcal{U}_{jk} = U_j \times V_k$ as in \eqref{eq:cov_B} and \eqref{eq:cov_R}, a family $\phi = (\phi_{jk})_{j \in J, k \in K}$ of continuous functions on $B \times \hat{B}$ is referred to as a \emph{bounded uniform partition of unity} subordinate to $(\mathcal{U}_{jk})$, if it satisfies
\begin{equation*}
  \supp \phi_{jk} \subset \mathcal{U}_{jk}, \qquad 0 \le \phi_{jk}(w,b) \le 1, \qquad \text{and} \qquad \sum_{j,k} \phi_{jk}(w,b) = 1
  \label{eq:covering_tU}
\end{equation*}
for all $j \in J$, $k \in K$ and $(w,b) \in B \times \hat{B}$.

\subsection{Approximation Operators}

We recall that every pair $X = (w, b)$ induces a joint translation
\begin{equation*}
  \tau_X : (z,a) \mapsto \tau_X (z,a) = (\varphi_w(z), t_b(a)),
  \label{}
\end{equation*}
where $t_b$ consists of a scalar translation $x \mapsto y - x$ in the first, and a unitary map in the second variable.

Let $\mathcal{U} \subset B \times \hat{B}$ be relatively compact with non-empty interior.
For $X, Y \in B \times \hat{B}$ let
\begin{equation*}
  \begin{split}
    \osc_\mathcal{U} (X,Y) &= \sup_{Z \in \mathcal{U}} \left| R (\tau_Z X, Y) - R(X, Y) \right| \\
    &= \sup_{Z \in \mathcal{U}} \left| \left< \rho(\tau_Z X) \psi - \rho(X) \psi, \rho(Y) \psi \right> \right|
  \end{split}
\end{equation*}
denote the \emph{oscillation} of $X$ and $Y$ with respect to $\mathcal{U}$.
The supremum is implicitly to be taken over all unitary maps acting on the third variable of $X$.
If $\mathcal{U}$ is reasonably symmetric, this supremum is the same as
\begin{equation*}
  \osc_\mathcal{U}(X, Y) = \sup_{Z \in \mathcal{U}} \left| \left< \rho(\tau_X Z) \psi - \rho(X) \psi, \rho(Y) \psi \right> \right|.
\end{equation*}

Let $\phi = (\phi_{jk})_{j \in J, k \in K}$ be a bounded uniform partition of unity, subordinate to a cover $(\mathcal{U}_{jk})_{j \in J, k \in K}$ with the aforementioned qualities.
We set $\mathcal{U} = U_o \times V_0$ and refer to this as the initial set of the partition.
Furthermore, let $(X_{jk})_{j,k}$ be chosen so that $\mathcal{U}_{jk} \subset \tau_{X_{jk}} \mathcal{U}$.
Then $\phi$ induces linear operators $T_\phi$ and $S_\phi$ defined by
\begin{align}
  T_\phi F(X) &= \sum_{j,k} \left< F, \phi_{jk} \right> R(X_{jk}, X) \nonumber \\
  &= \sum_{j,k} \iint_{B \times \hat{B}} F(Y) \phi_{jk}(Y) R(X_{jk}, X) \, \dd\xi(Y), \\
  S_\phi F(X) &= \sum_{j,k} \left< \phi_{jk}, R(X, \cdot) \right> F(X_{jk}) \nonumber \\
  &= \sum_{j,k} \iint_{B \times \hat{B}} \phi_{jk}(Y) R(Y, X) F(X_{jk}) \, \dd\xi(Y),
  \label{eq:approx_op_def}
\end{align}
whenever (unconditionally) convergent, meaning that the series are taken as limits of sums over finite subsets of $J \times K$, ordered by inclusion.

The practicality of these operators depends on the choice of the initial set $\mathcal{U}$, expressed by $\osc_\mathcal{U}$.
Let $C_\psi$ be defined as in \eqref{eq:kernel_int_cond}.

\begin{thm}
  \label{thm:atom_VM}
  Suppose there exists $\gamma < 1$ such that
  \begin{equation*}
    \iint_{B \times \hat{B}} \osc_\mathcal{U}(X,Y)\, \frac{m(X)}{m(Y)} \, \dd\xi(X)\quad \text{and} \quad  \iint_{B \times \hat{B}} \osc_\mathcal{U}(X,Y)\, \frac{m(X)}{m(Y)} \, \dd\xi(Y)
  \end{equation*}
  are uniformly bounded by $\gamma/C_\psi$.
  Then the operators $T_\phi$ and $S_\phi$ are bounded and invertible on every space $VM_m^p$.
\end{thm}

\begin{proof}
  Since $R$ is a reproducing kernel on $VM_m^p$, we obtain the decomposition
  \begin{align*}
    F(X) = \left< F, R(X, \cdot) \right> &= \iint F(Y) R(Y,X) \, \dd\xi(X) \\
    &= \sum_{j,k \in J \times K} \iint F(Y) \phi_{jk}(Y) R(Y,X) \, \dd\xi(Y)
  \end{align*}
  so that
  \begin{align*}
    \left| F(X) - T_\phi F(X) \right|
    &\le \sum_{j,k} \iint \left| F(Y) \right| \phi_{jk}(Y) \left| R(Y,X) - R(X_{jk}, X) \right| \dd\xi(Y).
  \end{align*}
  Now since $\supp \phi_{jk} \subset \mathcal{U}_{jk}$, we may assume $Y \in \mathcal{U}_{jk} \subset \tau_{X_{jk}} \mathcal{U}$, hence $Y = \tau_{X_{jk}} X'$ with $X' \in \mathcal{U}$, resulting in the estimate
  \begin{equation*}
    \left| R(Y,X) - R(X_{jk}, X) \right|
    \le \osc_\mathcal{U} (Y,X).
  \end{equation*}
  Using the weighted Young inequality and the integrability condition for $\osc_\mathcal{U}$,
  \begin{equation*}
    \left\Vert (\id - T_\phi)F \right\Vert_{\L_m^p} \le \gamma \left\Vert F \right\Vert_{\L_m^p}
  \end{equation*}
  follows, thus yielding $\left\Vert \id - T_\phi \right\Vert < 1$.
  Writing $T_\phi = \id -(\id - T_\phi)$,
  it follows that $T_\phi$ is a bounded operator with bounded inverse.
  
  Likewise, we have
  \begin{equation}
    \left| F(X) - S_\phi F(X) \right|
    \le \sum_{j,k} \iint_{B \times \hat{B}} \phi_{jk}(Y) \left| R(Y, X) \right| \left| F(Y) - F(X_{jk}) \right| \dd\xi(Y).
    \label{eq:S_phi_inv}
  \end{equation}
  Assuming $Y \in \tau_{X_{jk}} \mathcal{U}$ and using again the reproducing property, it follows that
  \begin{align*}
    \left| F(Y) - F(X_{jk}) \right| 
    &\le \iint \left| F(X') \right| \left| R(X',Y) - R(X',X_{jk}) \right| \dd\xi(X') \\
    &\le \iint \left| F(X') \right| \osc_\mathcal{U} (Y,X') \, \dd\xi(X'),
  \end{align*}
  and since $(\phi_{jk})$ is a partition of unity,
  \begin{align*}
    \sum_{j,k} \phi_{jk}(Y) \left| F(Y) - F(X_{jk}) \right|
    &\le \sum_{j,k} \phi_{jk}(Y) \iint \left| F(X') \right| \osc_\mathcal{U} (Y,X') \, \dd\xi(X') \nonumber \\
    &= \iint \left| F(X') \right| \osc_\mathcal{U} (Y,X') \, \dd\xi(X').
  \end{align*}
  Using the weighted Young inequality, we obtain from \eqref{eq:S_phi_inv}
  \begin{align*}
    \left\Vert F - S_\phi F \right\Vert_{\L_m^p}
    &\le C_\psi \Bigl\Vert \sum_{j,k} \left| F - F(X_{jk}) \right| \phi_{jk} \Bigr\Vert_{\L_m^p} \\
    &\le C_\psi \frac{\gamma}{C_\psi} \left\Vert F \right\Vert_{\L_m^p}.
  \end{align*}
  Consequently, $\left\Vert \id - S_\phi \right\Vert < 1$, so that $S_\phi$ has a bounded inverse.
\end{proof}

Corresponding representations for $f \in M_m^p$ can be obtained using the invertibility of $T_\phi, S_\phi$ and $V_\psi$.

\begin{cor}
  If $f \in M_m^p$, then $f$ can be decomposed as
  \begin{equation}
    f = \sum_{j,k} c_{jk}\, \rho(X_{jk}) \psi, \qquad \text{where} \quad c_{jk} = \big\langle T_\phi^{-1} V_\psi f, \phi_{jk} \big\rangle.
  \label{eq:decomp_atom}
  \end{equation}
  Moreover, $f$ can be reconstructed via
  \begin{equation}
    f = \sum_{j,k} \left< f, \rho(X_{jk}) \psi \right> e_{jk} =
    \sum_{j,k} V_\psi f(X_{jk}) e_{jk},
  \label{eq:reconstr_atom}
  \end{equation}
  where
  \begin{equation*}
    e_{jk} = V^*_\psi E_{jk}, \qquad \text{and} \quad E_{jk}(X) = S_\phi^{-1} \left< \phi_{jk}, R(X, \cdot) \right>
  \end{equation*}
\end{cor}

\begin{proof}
  Setting $F = V_\psi f \in VM_m^p$, we have
  \begin{equation*}
    F(X) = T_\phi T_\phi^{-1} F = \sum_{j,k} \big\langle T_\phi^{-1} F, \phi_{jk} \big\rangle R(X_{jk}, X),
  \end{equation*}
  and since $V^*_\psi V_\psi$ is the identity on $M_m^p$ and $V^*_\psi$ is continuous, we get
  \begin{equation*}
    f = V^*_\psi V_\psi f = V_\psi^* F = \sum_{j,k} \big\langle T_\phi^{-1} F, \phi_{jk} \big\rangle V^*_\psi R(X_{jk}, \cdot).
  \end{equation*}
  Regarding the last term on the right, we calculate for every $g \in \mathcal{S}(B)$,
  \begin{equation*}
    \begin{split}
      \big\langle V^*_\psi \left[ R(X_{jk}, \cdot) \right] , g \big\rangle
      &= \big\langle R(X_{jk}, \cdot) , V_\psi g \big\rangle \\
      &= \overline{\big\langle V_\psi g, R(X_{jk}, \cdot)  \big\rangle}
      = \overline{V_\psi g (X_{jk})}
      = \big\langle \rho(X_{jk}) \psi, g \big\rangle.
    \end{split}
  \end{equation*}
  This implies $V^*_\psi \left[ R(X_{jk}, \cdot) \right] = \rho(X_{jk}) \psi$, and the identity \eqref{eq:decomp_atom} follows.

  Similarly, since $S_\phi^{-1}$ is continuous, we may write
  \begin{equation*}
    \begin{split}
      F(X) = S_\phi^{-1} S_\phi F(X) &= \sum_{j,k} F(X_{jk}) S_\phi^{-1} \left< \phi_{jk} , R(X, \cdot) \right> \\
      &= \sum_{j,k} F(X_{jk}) E_{jk}(X).
    \end{split}
  \end{equation*}
  Recalling $F = V_\psi f$, we obtain by continuity of $V_\psi^*$
  \begin{equation*}
    \begin{split}
      f = V^*_\psi V_\psi f
      &= V^*_\psi \Bigl( \sum_{j,k} V_\psi f (X_{jk}) E_{jk} \Bigr) \\
      &= \sum_{j,k} V_\psi f(X_{jk}) V^*_\psi E_{jk}
      = \sum_{j,k} V_\psi f(X_{jk}) e_{jk}.
    \end{split}
  \end{equation*}
\end{proof}

\subsection{Frame Bounds}

It remains to determine under what condition the atomic decompositions on the spaces $M_m^p$ give rise to Banach frames.
With $c_{jk} = \bigl< T_\phi^{-1} V_\psi f, \phi_{jk} \bigr>$, respectively $c_{jk} = V_\psi f(X_{jk})$, this means
\begin{equation*}
  f \in M_m^p \quad \text{if and only if} \quad (c_{jk})_{j,k} \in \ell_m^p
\end{equation*}
and
\begin{equation*}
  A \left\Vert f \right\Vert_{M_m^p} \le \left\Vert (c_{jk})_{j,k} \right\Vert_{\ell_m^p} \le A' \left\Vert f \right\Vert_{M_m^p}
\end{equation*}
with $A, A'$ independent of $f$.
An appropriate weighted norm for sequences is given by
\begin{equation*}
  \left\Vert (c_{jk}) \right\Vert_{\ell_m^p} =  \Bigl( \sum_{j,k} \left| c_{jk} \right|^p m(X_{jk})^p \Bigr)^{1/p}
  = \Bigl\Vert \bigl( c_{jk} \, m(X_{jk}) \bigr)_{j,k} \Bigr\Vert_{\ell^p}
\end{equation*}
for $1 \le p < \infty$, and by
\begin{equation*}
  \left\Vert (c_{jk}) \right\Vert_{\ell_m^\infty} =  \sup_{j,k} \left| c_{jk} \right| m(X_{jk})
\end{equation*}
for $p = \infty$.

We assume throughout that the prerequisites of theorem \ref{thm:atom_VM} are satisfied, i.e., the operators $T_\phi$ and $S_\phi$ are automorphisms of $VM_m^p$.
We begin with \eqref{eq:decomp_atom}.

\begin{lem}
  Let $1 \le p < \infty$, then there exists $A' > 0$ so that if $f \in M_m^p$  and $c_{jk} = \big\langle T_\phi^{-1} V_\psi f, \phi_{jk} \big\rangle$, then $(c_{jk}) \in \ell_m^p$ and
  \begin{equation*}
    \left\Vert (c_{jk}) \right\Vert_{\ell_m^p} \le A' \left\Vert f \right\Vert_{M_m^p}.
  \end{equation*}
\end{lem}

\begin{proof}
  First, we prove the inequality
  \begin{equation}
    \left\Vert (\eta_{jk})_{j,k} \right\Vert_{\ell_m^p} \le C \Bigl\Vert \sum_{j,k} \left| \eta_{jk} \right| \mathbf{1}_{\mathcal{U}_{jk}} \Bigr\Vert_{\L_m^p}
    \label{eq:ineq_seq}
  \end{equation}
  for every sequence $(\eta_{jk}) \in \ell_m^p$, and some fixed constant $C$.
  Using the properties of the covering $(\mathcal{U}_{jk})$, we obtain
  \begin{equation*}
    \begin{split}
    \Bigl\Vert \sum_{j,k} \left| \eta_{jk} \right| \mathbf{1}_{\mathcal{U}_{jk}} \Bigr\Vert_{\L_m^p}^p
    &= \Bigl\Vert \sum_{r=1}^{r_0} \sum_{s=1}^{s_0} \sum_{(j,k) \in J_r \times K_s} \left| \eta_{jk} \right| \mathbf{1}_{\mathcal{U}_{jk}} \Bigr\Vert_{\L_m^p}^p \\
    &\ge \sum_{r=1}^{r_0} \sum_{s=1}^{s_0} \Bigl\Vert \sum_{(j,k)  \in J_r \times K_s} \left| \eta_{jk} \right| \mathbf{1}_{\mathcal{U}_{jk}} \Bigr\Vert_{\L_m^p}^p \\
    &= \sum_{r=1}^{r_0} \sum_{s=1}^{s_0} \sum_{(j,k)  \in J_r \times K_s} \bigl\Vert \left| \eta_{jk} \right| \mathbf{1}_{\mathcal{U}_{jk}} \bigr\Vert_{\L_m^p}^p.
    \end{split}
  \end{equation*}
  If $X \in \mathcal{U}_{jk}$, then $X = \tau_{X_{jk}} Y$ for some $Y \in \mathcal{U}$, and writing $m(z,b) = \kappa(z) v(b)$ with $\kappa$ and $v$ submultiplicative,
  \begin{equation*}
    m(X_{jk}) \le m(X)\, m(Y)
  \end{equation*}
  holds by \eqref{eq:kappa_perm} and \eqref{eq:v_perm}.
  By compactness of $\overline{\mathcal{U}}$,
\begin{equation*}
  \sum_{j,k} \bigl\Vert \left| \eta_{jk} \right| \mathbf{1}_{\mathcal{U}_{jk}} \bigr\Vert_{\L_m^p}^p
  \ge \sup_{Y \in \mathcal{U}}\, m(Y)^{-p}\, \xi(\mathcal{U}_{jk}) \sum_{j,k} \left| \eta_{jk} \right|^p m(X_{jk})^p,
\end{equation*}
the supremum is finite, and since $\xi(\mathcal{U}_{jk})$ has a positive lower bound, \eqref{eq:ineq_seq} follows.

Let $F \in \L_m^p$, then by the above inequality,
\begin{equation*}
  \left\Vert \left( \big< \left| F \right| , \phi_{jk} \big> \right)_{j,k} \right\Vert_{\ell_m^p}
  \le C \Bigl\Vert \sum_{j,k} \left< \left| F \right| , \phi_{jk} \right> \mathbf{1}_{\mathcal{U}_{jk}} \Bigr\Vert_{\L_m^p}.
\end{equation*}
Next, set $\mathcal{J}_X = \left\{ (j,k) \in J \times K : X \in \mathcal{U}_{jk} \right\}$. Then $|\mathcal{J}_X| \le r_0 s_0$ and
\begin{equation*}
  \sum_{j,k} \left< \left| F \right| , \phi_{jk} \right> \mathbf{1}_{\mathcal{U}_{jk}} (X)
  = \sum_{(j,k) \in \mathcal{J}_X} \left< \left| F \right| , \phi_{jk} \right> \le \left< \left| F \right| , Q(X, \cdot) \right>,
\end{equation*}
where
\begin{equation*}
  Q (X,Y) = \sum_{(j,k) \in \mathcal{J}_X} \mathbf{1}_{\mathcal{U}_{jk}}(Y)
  = \sum_{(j,k) \in \mathcal{J}_Y} \mathbf{1}_{\mathcal{U}_{jk}}(X)= Q(Y,X).
\end{equation*}
Therefore, $Q(X,Y) > 0$ holds precisely if for some $j,k$, both $X, Y \in \mathcal{U}_{jk}$.
Furthermore, for each two such points there exist $X', Y' \in \mathcal{U}$ with $X = \tau_{X_{jk}} X'$ and $Y = \tau_{X_{jk}} Y'$, which implies
\begin{equation*}
  m(X) \le m(X_{jk}) m(X') \quad \text{and} \quad m(Y) \ge m(X_{jk}) / m(Y'),
\end{equation*}
again by \eqref{eq:kappa_perm} and \eqref{eq:v_perm}.
By compactness of $\overline{\mathcal{U}}$ we conclude
\begin{equation}
  \frac{m(X)}{m(Y)} \le m(X')m(Y') \le C_0
  \label{eq:weight_quot}
\end{equation}
with $C_0$ independent of $X$ and $Y$.
Thus,
\begin{equation}
  \iint Q(X, Y) \frac{m(X)}{m(Y)} \, \dd\xi(Y)
  \le C_0 r_0 s_0\, \xi(\mathcal{U}_{jk})
  \label{eq:count_kernel}
\end{equation}
is bounded, since $\xi(\mathcal{U}_{jk})$ is bounded by a constant independent of $j,k$.
Similarly, the same is true for
\begin{equation*}
  \iint Q(X, Y) \frac{m(X)}{m(Y)} \, \dd\xi(X).
\end{equation*}
The weighted Young inequality is thus applicable and yields the estimate
\begin{equation*}
  \left\Vert \bigl( \left< F , \phi_{jk} \right> \bigr)_{j,k} \right\Vert_{\ell_m^p} \le C \left\Vert \left< \left| F \right| , Q(X, \cdot) \right> \right\Vert_{\L_m^p(X)}
  \le \tilde{C} \left\Vert F \right\Vert_{\L_m^p}.
\end{equation*}
Setting $F = T_\phi^{-1} V_\psi f$, we finally obtain
\begin{align*}
  \left\Vert \bigl( \bigl< T_\phi^{-1} V_\psi f, \phi_{jk} \bigr> \bigr)_{j,k} \right\Vert_{\ell_m^p} 
  &\le \tilde{C} \big\Vert T_\phi^{-1} V_\psi f \big\Vert_{\L_m^p} \nonumber \\
  &\le \tilde{C} \big\Vert T_\phi^{-1} \big\Vert \left\Vert V_\psi f \right\Vert_{\L_m^p}
  = \tilde{C} \bigl\Vert T_\phi^{-1} \bigr\Vert \left\Vert f \right\Vert_{M_m^p}.
\end{align*}
\end{proof}

\begin{thm}
  Suppose $R$ satisfies the following condition: There exists a constant $C_{\mathcal{U}}$, so that for all $X \in B \times \hat{B}$,
  \begin{equation}
    \iint \sup_{Z \in \mathcal{U}} \left| R(\tau_Z Y, X) \right| \frac{m(X)}{m(\tau_Z Y)}\, \dd\xi(Y) \le C_{\mathcal{U}}.
    \label{eq:framebounds_ass1}
  \end{equation}
  Let $1 \le p \le \infty$ and $(c_{jk})_{j,k} \in \ell_{m}^p$.
  Then $f = \sum_{j,k} c_{jk}\, \rho(X_{jk}) \psi \in M_{m}^p$ and
  \begin{equation}
    A \left\Vert f \right\Vert_{M_{m}^p} \le \left\Vert (c_{jk}) \right\Vert_{\ell_{m}^p}
    \label{eq:framebounds_Ml}
  \end{equation}
  with $A > 0$ independent of $f$.
\end{thm}

\begin{proof}
  We have to show that the operator
  \begin{equation*}
    T: \ell_{m}^p \to \L_{m}^p, \quad (c_{jk}) \mapsto \sum_{j,k} c_{jk} R(X_{jk}, \cdot)
  \end{equation*}
  is bounded for all $1 \le p \le \infty$.
  By the weighted Riesz-Thorin theorem \ref{thm:riesz_thorin}, it is enough to prove this statement for $p = 1$ and $p = \infty$.
  If $p = 1$, then
  \begin{equation*}
    \begin{split}
      \Bigl\Vert \sum_{j,k} c_{jk} R(X_{jk}, \cdot) \Bigr\Vert_{\L_m^1}
      &\le \sum_{j,k} \bigl| c_{jk} \bigr| \bigl\Vert R(X_{jk}, \cdot) \bigr\Vert_{\L_m^1} \\
      &= \sum_{j,k} \bigl| c_{jk} \bigr| m(X_{jk}) \biggl\Vert \frac{R(X_{jk}, \cdot)}{m(X_{jk})} \biggr\Vert_{\L_m^1} \\
      &\le C_\psi \bigl\Vert (c_{jk}) \bigr\Vert_{\ell_m^1}.
    \end{split}
  \end{equation*}
  For $p = \infty$, we obtain the estimate
  \begin{align}
    \Bigl\Vert \sum_{j,k} c_{jk} R(X_{jk}, \cdot) \Bigr\Vert_{\L_m^\infty}
    &= \sup_{X \in B \times \hat{B}} \sum_{j,k} \bigl| c_{jk} \bigr| \bigl| R(X_{jk}, X) \bigr| m(X) \nonumber \\
    &\le \sup_{j,k}\, \bigl| c_{jk} \bigr| m(X_{jk}) \sup_{X \in B \times \hat{B}} \sum_{j,k}  \bigl| R(X_{jk}, X) \bigr| \frac{m(X)}{m(X_{jk})} \nonumber \\
    &= \bigl\Vert (c_{jk}) \bigr\Vert_{\ell_m^\infty} \sup_{X \in B \times \hat{B}} \sum_{j,k}  \bigl| R(X_{jk}, X) \bigr| \frac{m(X)}{m(X_{jk})}.
    \label{eq:framebounds_W}
  \end{align}
  Regarding the last term on the right side, we note that
  \begin{align*}
  \left| R(X_{jk}, X) \right| \frac{m(X)}{m(X_{jk})} 
  &\le \sup_{Z \in \mathcal{U}} \left| R(\tau_Z Y, X) \right| \frac{m(X)}{m(\tau_Z Y)}
  \end{align*}
  holds whenever $Y \in \mathcal{U}_{jk}$.
  Integration turns this into
  \begin{equation*}
    \iint_{\mathcal{U}_{jk}} \sup_{Z \in \mathcal{U}} \left| R(\tau_Z Y, X) \right| \frac{m(X)}{m(\tau_Z Y)} \, \dd\xi(Y) \ge \left| R(X_{jk}, X) \right| \frac{m(X)}{m(X_{jk})} \xi(\mathcal{U}_{jk})
  \end{equation*}
  for all $j, k$.
  Fix $1 \le r \le r_0$ and $1 \le s \le s_0$, then using \eqref{eq:framebounds_ass1} and the above, we obtain
  \begin{align*}
    C_\mathcal{U} &\ge \iint\limits_{B \times \hat{B}} \sup_{Z \in \mathcal{U}} \left| R(\tau_Z Y, X) \right| \frac{m(X)}{m(\tau_Z Y)}\, \dd\xi(Y) \\
    &\ge \sum_{(j,k) \in J_r \times K_s} \iint_{\mathcal{U}_{jk}} \sup_{Z \in \mathcal{U}} \left| R(\tau_Z Y, X) \right| \frac{m(X)}{m(\tau_Z Y)}\, \dd\xi(Y) \nonumber \\
    &\ge \sum_{(j,k) \in J_r \times K_s} \left| R(X_{jk},X) \right| \frac{m(X)}{m(X_{jk})} \xi(\mathcal{U}_{jk}),
  \end{align*}
  which implies
  \begin{equation*}
    \sum_{(j,k) \in J_r\times K_s} R(X_{jk},X) \frac{m(X)}{m(X_{jk})} \le \frac{C_\mathcal{U}}{\rho(\mathcal{U}_{jk})} \le \tilde{C}_\mathcal{U}
  \end{equation*}
  with $\tilde{C}_{\mathcal{U}}$ independent of $j,k$ (and $r,s$), since $\xi(\mathcal{U}_{jk})$ is uniformly bounded from below.
  Returning to \eqref{eq:framebounds_W}, we obtain
  \begin{equation*}
    \Bigl\Vert \sum_{j,k} c_{jk} R(X_{jk}, \cdot) \Bigr\Vert_{\L_m^\infty}
    \le r_0 s_0 \tilde{C}_\mathcal{U} \left\Vert (c_{jk}) \right\Vert_{\ell_m^\infty}.
  \end{equation*}

  Finally, by continuity of $V_\psi$,
  \begin{align*}
    \sum_{j,k} c_{jk} R(X_{jk},X) &= \sum_{j,k} c_{jk} V_\psi[\rho(X_{jk})\psi](X) = V_\psi \Bigl[ \sum_{j,k} c_{jk} \rho(X_{jk})\psi \Bigr] (X)
  \end{align*}
  and the above inequality may be expressed as
  \begin{equation*}
    \Bigl\Vert \sum_{j,k} c_{jk} \rho(X_{jk})\psi  \Bigr\Vert_{M_m^\infty} = \Bigl\Vert \sum_{j,k} c_{jk} R(X_{jk},\cdot) \Bigr\Vert_{\L_m^\infty}
    \le \frac{1}{A} \left\Vert (c_{jk}) \right\Vert_{\ell_m^\infty}.
  \end{equation*}
\end{proof}

Now we focus on the frame bounds for the decomposition in \eqref{eq:reconstr_atom}.

\begin{thm}
  Let $f \in M_m^p$. Then $\bigl( V_\psi f (X_{jk}) \bigr)_{j,k} \in \ell_m^p$ and
  \begin{equation*}
    \left\Vert \bigl( V_\psi f (X_{jk}) \bigr)_{j,k} \right\Vert_{\ell_m^p} \le A' \left\Vert f \right\Vert_{M_m^p}
  \end{equation*}
  with $A' > 0$ independent of $f$.
\end{thm}

\begin{proof}
  Set $F = V_\psi f$, then the above inequality is equivalent to
  \begin{equation*}
    \left\Vert \bigl( F(X_{jk}) \bigr)_{j,k} \right\Vert_{\ell_m^p} \le A' \left\Vert F \right\Vert_{\L_m^p}.
  \end{equation*}

  For $p = 1$, we obtain the estimate
  \begin{align*}
    \sum_{j,k} \left| F(X_{jk}) \right| m(X_{jk})
    &= \sum_{j,k} \left| \left< F, R(X_{jk}, \cdot) \right> \right| m(X_{jk}) \\
    &\le \sum_{j,k} \iint \left| F(X) \right| \left| R(X_{j,k},X) \right| m(X_{jk}) \, \dd\xi(X) \\
    &\le \iint \left| F(X) \right| m(X) \sum_{j,k} \left| R(X_{jk}, X) \right| \frac{m(X_{jk})}{m(X)} \, \dd\xi(X) \\
    &\le \left\Vert F \right\Vert_{\L_m^1} \sup_{X \in B\times \hat{B}} \sum_{j,k} \left| R(X_{jk}, X) \right| \frac{m(X_{jk})}{m(X)}.
  \end{align*}
  As before, the supremum on the right is bounded by a constant $r_0 s_0 \tilde{C}_{\mathcal{U}}$.

If $p = \infty$, then
\begin{align*}
  \sup_{j,k} \left| F(X_{jk}) \right| m(X_{jk})
  &= \sup_{j,k} \left| \left< F, R(X_{jk}, \cdot \right> \right| m(X_{jk}) \\
  &\le \sup_{j,k} \int \left| F(X) \right| \left| R(X_{jk},X) \right| m(X_{jk}) \, \dd\xi(X) \\
  &\le \sup_{j,k} \int \left| R(X_{jk},X) \right| \frac{m(X_{jk})}{m(X)} \, \dd\xi(X)\ \sup_{Y} \left| F(Y) \right| m(Y) \\
  &\le C_\psi \left\Vert F \right\Vert_{\L_m^\infty}.
\end{align*}
The general result follows again by the weighted Riesz-Thorin theorem.
\end{proof}

\begin{thm}
  Let $f$ be a distribution.
  If $\bigl( V_\psi f (X_{jk}) \bigr)_{j,k} \in \ell_m^p$, then $f \in M_m^p$, and
  \begin{equation*}
    A \left\Vert f \right\Vert_{M_m^p} \le \Bigl\Vert \bigl( V_\psi f(X_{jk}) \bigr)_{j,k} \Bigr\Vert_{\ell_m^p},
  \end{equation*}
  with $A$ independent of $f$.
\end{thm}

\begin{proof}
  We show that the operator
  \begin{equation*}
    (c_{jk}) \mapsto F, \quad F(X) = \Bigl< \sum_{j,k} c_{jk} \phi_{jk}, R(X, \cdot) \Bigr>
  \end{equation*}
  is bounded from $\ell_m^p$ to $\L_m^p$ for all $1 \le p \le \infty$, again by using the weighted Riesz-Thorin theorem.

  Let $(c_{jk}) \in \ell_m^1$.
  Using the weighted Young inequality, we obtain the estimate
  \begin{align*}
    \Bigl\Vert \Bigl< \sum_{j,k} c_{jk} \phi_{jk} , R(X, \cdot) \Bigr> \Bigr\Vert_{\L_m^1(X)}
    &\le C_\psi \Bigl\Vert \sum_{j,k} c_{jk} \phi_{jk} \Bigr\Vert_{\L_m^1} \\
    &\le C_\psi \sum_{j,k} \iint \left| c_{jk} \right| m(X_{jk})\, \phi_{jk}(X) \frac{m(X)}{m(X_{jk})} \, \dd\xi(X) \\
    &\le C_\psi \left\Vert (c_{jk}) \right\Vert_{\ell_m^1} \sup_{j,k} \iint_{\mathcal{U}_{jk}} \frac{m(X)}{m(X_{jk})}\, \dd\xi(X).
  \end{align*}
  The integral term on the right is bounded uniformly in $j,k$, with a similar argument as in \eqref{eq:weight_quot} and \eqref{eq:count_kernel}.

  The estimate for $(c_{jk}) \in \ell_m^\infty$ is also obtained using the weighted Young inequality,
  \begin{align*}
    \Bigl\Vert \Bigl< \sum_{j,k} c_{jk} \phi_{jk} , R(X, \cdot) \Bigr> \Bigr\Vert_{\L_m^\infty(X)}
    &\le C_\psi \sup_X \Bigl| \sum_{j,k} c_{jk} \phi_{jk}(X) \Bigr| m(X) \\
    &\le C_\psi \sup_{(j,k)} \left| c_{jk} \right| m(X_{jk}) \cdot \sup_X \sum_{j,k} \phi_{jk} (X) \frac{m(X)}{m(X_{jk})} \\
    &\le C_\psi C \left\Vert (c_{jk}) \right\Vert_{\ell_m^\infty},
  \end{align*}
  with some constant $C$, since at most $r_0 s_0$ many $\phi_{jk}(X)$ are nonzero.
    
  Next, $(c_{jk}) \in \ell_m^p$ implies $\sum c_{jk} \phi_{jk} \in \L_m^p$, and by continuity of $V_\psi$ and $V_\psi^*$, we have
  \begin{equation*}
    \Bigl< \sum_{j,k} c_{jk} \phi_{jk}, R(X, \cdot) \Bigr> = \sum_{j,k} c_{jk} \bigl< \phi_{jk}, R(X, \cdot) \bigr> \in VM_m^p.
  \end{equation*}
  Since $S_\phi^{-1}$ is continuous, the operator mapping sequences $(c_{jk})$ to functions
  \begin{equation}
    \label{eq:VM_unique}
    X \mapsto S_\phi^{-1} \sum_{j,k} c_{jk} \bigl< \phi_{jk}, R(X, \cdot) \bigr>
    = \sum_{j,k} c_{jk} S_\phi^{-1} \bigl< \phi_{jk}, R(X, \cdot) \bigr>
  \end{equation}
  is thus bounded from $\ell_m^p$ to $VM_m^p$.
  Now, set $c_{jk} = V_\psi f(X_{jk}) = \left< f, \rho(X_{jk}) \psi \right>$ and assume $(c_{jk}) \in \ell_m^p$.
  For this sequence, the function defined in \eqref{eq:VM_unique} is unique in $VM_m^p$, so it must be equal to $V_\psi f$.
  By \eqref{eq:reconstr_atom} and the above calculation, we obtain thus
  \begin{equation*}
    \left\Vert f \right\Vert_{M_m^p} = \left\Vert V_\psi f \right\Vert_{\L_m^p} = \Bigl\Vert \sum_{j,k} c_{jk} S_\phi^{-1} \left< \phi_{jk}, R(X, \cdot) \right> \Bigr\Vert_{\L_m^p} \le \frac{1}{A} \left\Vert (c_{jk})_{j,k} \right\Vert_{\ell_m^p}
  \end{equation*}
  with $A$ independent of $f$.
\end{proof}

\subsection{Nonlinear Approximation}

The approximations obtained so far can be summarized as follows.
Given a frame $\left\{ \psi_{jk} : (j,k) \in J \times K \right\}$ for $M_m^p$, every $f$ admits a suitable sequence $c = (c_{jk})_{j,k}$ with unconditional convergence
\begin{equation}
  \label{eq:atom_simplified}
  f = \sum_{j,k} c_{jk} \psi_{jk}
\end{equation}
and satisfying
\begin{equation}
  \label{eq:framebounds_simplified}
  A \left\Vert f \right\Vert_{M_m^p} \le \left\Vert c \right\Vert_{\ell_m^p} \le A' \left\Vert f \right\Vert_{M_m^p}.
\end{equation}
Due to the non-orthogonal and redundant nature of Banach frames, approximation of $f$ in terms of \emph{finite} sums is not a linear process.
More precisely, given $N \in \N$ and
\begin{equation*}
  \Sigma_N = \Bigl\{ S = \sum_{(j,k) \in \mathcal{I}} b_{jk} \psi_{jk} : \mathcal{I} \subset J \times K,\ \left| \mathcal{I} \right| \le N \Bigr\},
\end{equation*}
we are interested in the quality of the \emph{best $N$-term approximation}, given by the error
\begin{equation*}
  E_N(f)_{M_m^p} = \inf_{S \in \Sigma_N} \left\Vert f - S \right\Vert_{M_m^p}.
\end{equation*}
We shall need the following lemma, see \cite{da_st_te1}, \cite{da_st_te3}, \cite{groechenig_samarah}.

\begin{lem}
  \label{lem:nonlin_approx_lem}
  Let $a = (a_j)_{j \in \N}$ be a decreasing sequence of positive numbers.
  For $p, q > 0$ set $\alpha = 1/p - 1/q$ and $E_{N,q}(a) = \bigl( \sum_{j=N}^\infty a_j^q \bigr)^{1/q}$.
  Then for $0 < p < q \le \infty$, the inequalities
  \begin{equation*}
    2^{-1/p} \left\Vert a \right\Vert_{\ell^p} \le \biggl( \sum_{N=1}^\infty \frac{1}{N} \bigl( N^\alpha E_{N,q}(a) \bigr)^p \biggr)^{1/p} \le C \left\Vert a \right\Vert_{\ell^p}
  \end{equation*}
  hold with $C > 0$ depending only on $p$.
\end{lem}

The asymptotic behavior of the error is answered by the next theorem, see again \cite{da_st_te3}.
\begin{thm}
  \label{thm:nonlin_approx}
  Let $1 \le p \le \infty$ and let $\{ \psi_{jk} : (j,k) \in J \times K \}$ be a frame for $M_m^p$.
  Let $p < q$, set $\alpha = 1/p - 1/q$, and let $f \in M_m^p$, then we have
  \begin{equation*}
    \biggl( \sum_{N=1}^\infty \frac{1}{N} \bigl( E_N(f)_{M_m^q} \bigr)^p \biggr)^{1/p} \le C \left\Vert f \right\Vert_{M_m^p}
  \end{equation*}
  for a constant $C > 0$.
\end{thm}

\begin{proof}
  Write $f = \sum_{j,k} c_{jk} \psi_{jk}$ and let $\pi: \N \to J \times K$ be an enumeration of the index set such that
  \begin{equation*}
    \bigl| c_{\pi(1)}\, m \!\left(X_{\pi(1)} \right) \bigr| \ge \bigl| c_{\pi(2)}\, m \!\left( X_{\pi(2)} \right) \bigr| \ge \dots
  \end{equation*}
  is decreasing.
  Then we have
  \begin{equation*}
    E_N(f)_{M_m^q} \le \biggl\Vert \sum_{j = N+1}^\infty c_{\pi(j)}\, \psi_{\pi(j)} \biggr\Vert_{M_m^q},
  \end{equation*}
  and setting $\tilde{c} = ( \tilde{c}_j)_{j \in \N} = \left( c_{\pi(j)} m(X_{\pi(j)}) \right)_{j\in\N}$, we obtain by \eqref{eq:framebounds_simplified} that
  \begin{align*}
    E_N(f)_{M_m^q} &\le C' \biggl( \sum_{j = N+1}^\infty \left| \tilde{c}_j \right|^q \biggr)^{1/q}
    = C'\, E_{N+1,q} ( \tilde{c} )
    \le C'\, E_{N,q} ( \tilde{c} ).
  \end{align*}
  Applying lemma \ref{lem:nonlin_approx_lem} and \eqref{eq:framebounds_simplified} yields now
  \begin{align*}
    \biggl( \sum_{N=1}^\infty \frac{1}{N} \bigl( N^\alpha E_N(f) \bigr)^p \biggr)^{1/p}
    &\le \biggl( \sum_{N=1}^\infty \frac{1}{N} \bigl( N^\alpha C'\, E_{N,q} ( \tilde{c} ) \bigr)^p \biggr)^{1/p} \\
    &\le C'' \left\Vert \tilde{c} \right\Vert_{\ell^p}
    = C'' \left\Vert c \right\Vert_{\ell_m^p} \\
    &\le C \left\Vert f \right\Vert_{M_m^p}. \qedhere
  \end{align*}
\end{proof}

\appendix
\section{Appendix}

The following is a weighted version of the Young inequality, the proof of which can be found in \cite{da_st_te2}.
\begin{thm}
  \label{thm:young}
  Let $(X, \mathcal{A}, \nu)$ be a $\sigma$-finite measure space and $m : X \to \R^+$ a continuous weight function.
  Let $K$ be an $\mathcal{A} \otimes \mathcal{A}$-measurable function for which the integrals
  \begin{equation*}
    \int_X \left| K(x,y) \right| \frac{m(x)}{m(y)}\, \dd\nu(y) \quad \text{and} \quad \int_X \left| K(x,y) \right| \frac{m(x)}{m(y)}\, \dd\nu(x)
  \end{equation*}
  are bounded almost everywhere by a constant $C$.
  If $1 \le p \le \infty$, the operator $T$ given by
  \begin{equation*}
    Tf(x) = \int_X K(x,y) f(y) \, \dd\nu(y)
  \end{equation*}
  is well defined for $f \in \L_m^p$ and
  \begin{equation*}
    \left\Vert Tf \right\Vert_{\L_m^p} \le C \left\Vert f \right\Vert_{\L_m^p}.
  \end{equation*}
\end{thm}

We also include a version of the Riesz-Thorin theorem for weighted function spaces, see again \cite{da_st_te2}.
Here $\L_m^p$ and $\ell_m^p$ are interchangeable.

\begin{thm}
  \label{thm:riesz_thorin}
  Let $T$ be a bounded linear operator from $\L_m^1$ into $\ell_m^1$ with norm $t_1$ and from $\L_m^\infty$ into $\ell_m^\infty$ with norm $t_\infty$.
  Then, for every $1 < p < \infty$, the operator $T$ is also bounded from $\L_m^p$ into $\ell_m^p$ with norm $t_1^{1/p} t_\infty^{(p-1)/p}$.
\end{thm}

\subsection{The Helgason transform}

For convenience, we give an outline of some points needed along the way of deriving the Helgason-Fourier transform as an $\L^2$-isometry.
We begin with two properties of the Poisson kernel $P_{\lambda, \zeta}$ (cf.\ chapter 3 in \cite{rudin_ball}).

\begin{prop}
  \label{prop:add_thm}
  \hspace{0ex}
  \begin{enumerate}[(i)]
    \item Let $g \in \L^1(S,\sigma)$ and $z \in B$. Then
      \begin{equation}
        \int_S g(\varphi_z(\zeta))\, \dd\sigma(\zeta)
	= \int_S P_{0,\zeta}(z)^2 g(\zeta)\, \dd\sigma(\zeta).
        \label{eq:poiss_inv}
      \end{equation}
    \item (Addition theorem) For all $\lambda \in \C$, $\zeta \in S$ and $z,w \in B$,
      \begin{equation}
	P_{\lambda,\zeta}(\varphi_w(z)) = P_{\lambda, \zeta}(w) P_{\lambda, \varphi_w(\zeta)}(z).
        \label{eq:add_thm}
      \end{equation}
  \end{enumerate}
\end{prop}

Let $f * g$ denote the commutative convolution of $f$ and $g$ on $B$ given by
\begin{equation}
  \label{eq:convolution_B}
  (f * g)(z) = \int_B f(w) g(\varphi_z(w)) \, \dd\mu(w).
\end{equation}
Then, as a consequence of \eqref{eq:add_thm},
\begin{equation}
  \widehat{(f * g)}(\lambda,\zeta) = \hat{f}(\lambda,\zeta) \hat{g}(\lambda)
  \label{eq:FT_conv}
\end{equation}
holds for every $f \in \D(B)$ and $g \in \D_\natural(B)$.

For a function $f \in \D(B)$ its \emph{radialization} is defined as
\begin{equation*}
  f^\natural (r) = \int_S f(r\omega) \, \dd\sigma(\omega)
  = \int_{\U(n)} f(u(r\zeta)) \, \dd u.
\end{equation*}
Fix $z \in B$ and set $f_z = (f \circ \varphi_z)^\natural$.
Then $f_z \in \D_\natural(B)$ and the inversion formula \eqref{eq:sph_inversion} for the spherical transform reads
\begin{equation}
  f(z) = f_z(o) = \frac{1}{2} \int_\R \hat{f_z}(\lambda) \left| \hc(\lambda) \right|^{-2} \dd\lambda
  \label{eq:helg_pre_inv}
\end{equation}
since $\phi_\lambda(o) = 1$.
Moreover, since $\mu$ is invariant under $\U(n)$, we have
\begin{align*}
  \int_S \hat{f}(\lambda,\zeta) \, \dd\sigma(\zeta)
  &= \int_B f(w) \phi_\lambda(w) \, \dd\mu(w) \\
  &= \int_B \int_{\U(n)} f(w) \phi_\lambda(u(w)) \, \dd u \, \dd\mu(w) \\
  &= \int_B f^\natural (w) \phi_\lambda(w) \dd\mu(w)
  = \widehat{f^\natural}(\lambda).
\end{align*}
Replacing $f$ by $(f \circ \varphi_z)$ we obtain by Proposition \ref{prop:add_thm}
\begin{align*}
  \hat{f_z}(\lambda) = (f \circ \varphi_z)\widehat{\,} (\lambda)
  &= \int_B f(\varphi_z(w)) \phi_{-\lambda}(w) \, \dd\mu(w) \\
  &= \int_S \int_B f(\varphi_z(w)) P_{-\lambda,\zeta}(w)\, \dd\mu(w)\, \dd\sigma(\zeta) \\
  &= \int_S \int_B f(w) P_{-\lambda,\zeta}(\varphi_z(w))\, \dd\mu(w)\, \dd\sigma(\zeta) \\
  &= \int_S \int_B f(w) P_{-\lambda,\varphi_z(\zeta)}(w) P_{-\lambda,\zeta}(z) \, \dd\mu(w)\, \dd\sigma(\zeta) \\
  &= \int_S \int_B f(w) P_{-\lambda,\zeta}(w) P_{-\lambda,\zeta}(z) P_{0,\zeta}(z)^2\, \dd\mu(w)\, \dd\sigma(\zeta) \\
  &= \int_S \int_B f(w) P_{-\lambda,\zeta}(w) P_{\lambda,\zeta}(z)\, \dd\mu(w)\, \dd\sigma(\zeta) \\
  &= \int_S \hat{f}(\lambda,\zeta) P_{\lambda,\zeta}(z)\, \dd\sigma(\zeta).
\end{align*}
Returning to \eqref{eq:helg_pre_inv}, we obtain the inversion formula
\begin{equation}
  f(z) = \frac{1}{2} \int_\R \int_S \hat{f}(\lambda, \zeta) P_{\lambda,\zeta} (z) \, \dd\sigma(\zeta) \left| \hc(\lambda) \right|^{-2} \dd\lambda
\label{eq:helg_inv}
\end{equation}
for the Helgason-Fourier transform.

\begin{prop}
  \label{thm:F_orth}
  \hspace{0ex}
  \begin{enumerate}[(i)]
    \item Let $\lambda \notin i\Z_{\ge n}$ and $g \in \L^2(S,\sigma)$.
      Then
    \begin{equation}
      \int_S g(\zeta) P_{\lambda,\zeta}(z) \, \dd\sigma(\zeta) = 0 \qquad \text{for all $z \in B$}
      \label{eq:F_orth}
    \end{equation}
    implies $g = 0$.
    \item Let $-\lambda \notin i\Z_{\ge n}$. Then the functions
      \begin{equation*}
        \zeta \mapsto \hat{f}(\lambda,\zeta), \qquad f \in \D(B)
      \end{equation*}
      form a dense subset of $\L^2(S,\sigma)$.
  \end{enumerate}
\end{prop}

\begin{proof}
  \textit{(i)} We substitute $z = (\tanh r) \omega$ ($r \in \R$, $\omega \in S$), and $\alpha = (n+i\lambda)/2$.
  Then \eqref{eq:F_orth} is equivalent to
  \begin{equation*}
    \int_S \left( \cosh r - (\sinh r)\! \left< \omega, \zeta \right>  \right)^{-2\alpha} F(\zeta)\, \dd\sigma(\zeta) = 0,
  \end{equation*}
  which holds independently of $r$.
  A function $h$ on $S$ shall be called $F$-orthogonal, if $\int_S h \, F\, \dd\sigma = 0$.
  Setting
  \begin{equation*}
    g_p(\omega, \zeta) = \Bigl[ \partial_r^p \bigl( \cosh  r - (\sinh r) \! \left< \omega, \zeta \right> \bigr)^{-2\alpha} \Bigr]_{r = 0},
  \end{equation*}
  the functions $\zeta \mapsto g_p(w, \zeta)$ are $F$-orthogonal for every $p \in \N_0$ and $\omega \in S$.
  Now $g_p$ can be calculated using
  \begin{equation*}
    \partial_r^k \left( \cosh r - (\sinh r) z \right)^{-\alpha} = \sum_{j=0}^k c_j^k \left( \cosh r - (\sinh r) z \right)^{-\alpha-j} \left( \sinh r - (\cosh r) z \right)^j,
  \end{equation*}
  where $c_j^k$ are dependent on $\alpha, j, k$ and $c_k^k = (-1)^k \fup{\alpha}{k} \ne 0$, and the Leibniz rule, resulting in a sum
  \begin{equation*}
    g_p (\omega, \zeta) = \sum_{k=0}^p \sum_{l,j \in \N_0} c_l^{p-k} c_j^{k} (-1)^{l+j} \left< \omega, \zeta \right>^l \overline{\left< \omega, \zeta \right>}^j
  \end{equation*}
  (setting $c_j^k = 0$ if $j > k$).
  We thus obtain $g_0(\omega, \zeta) = 1$ and $g_1(\omega, \zeta) = -\alpha \bigl(\left< \omega, \zeta \right> + \overline{\left< \omega, \zeta \right>} \bigr)$.
  We observe that $\left< \omega, \zeta \right>$ and $\overline{\left< \omega, \zeta \right>}$ are both linear combinations of $g_1(\omega,\zeta)$ and $g_1(i\omega, \zeta)$, and thereby $F$-orthogonal.
  In fact, every monomial
  \begin{equation*}
    \left< \omega, \zeta \right>^{p-k} \overline{\left< \omega, \zeta \right>}^{k}, \quad 0 \le k \le p,\, \omega \in S
  \end{equation*}
  is a linear combination of the functions $g_p(e^{2i\pi/2^j} \omega, \zeta)$, $j = 0, \dots, p$.
  By induction on $p$, these generate an $F$-orthogonal self-adjoint algebra, which separates the points in $S$, and therefore lies dense in $C(S)$ by the Stone-Weierstrass theorem.
  As a consequence, $F = 0$.

  \textit{(ii)} Suppose $F \in \L^2(S,\sigma)$ satisfies
  \begin{equation*}
    0 = \int_S \hat{f}(\lambda, \zeta) F(\zeta) \, \dd\sigma(\zeta) = \int_B f(z) \biggl( \int_S P_{-\lambda,\zeta}(z) F(\zeta)\, \dd\sigma(\zeta) \biggr) \dd\mu(z),
  \end{equation*}
  then the inner integral vanishes, since $\D(B)$ lies dense in $\L^2(B,\mu)$.
  By \textit{(i)}, $F = 0$.
\end{proof}

\begin{thm}
  \label{thm:FT_isometry}
  The Helgason Fourier transform extends to an isometry
  \begin{equation}
    \L^2(B,\mu) \longrightarrow \L^2(\R^+ {\times}\, S, \left| \hc(\lambda) \right|^{-2} \dd\sigma\, \dd\lambda)
    \label{eq:FT_isometry}
  \end{equation}
\end{thm}

\begin{proof}
  Let $F \in \L^2(\R^+{\times}\, S)$ satisfy
  \begin{equation*}
    \int_{\R^+} \int_S \hat{f}(\lambda, \zeta) F(\lambda, \zeta) \left| \hc(\lambda) \right|^{-2} \dd\sigma(\zeta)\, \dd\lambda = 0.
  \end{equation*}
  for every $f \in \D(B)$.
  Then the equation remains true if $f$ is replaced by $f * g$ ($g \in \D_\natural(B)$).
  By \eqref{eq:FT_conv} we obtain
  \begin{equation*}
    \int_{\R^+} \hat{g}(\lambda) \int_S \hat{f}(\lambda,\zeta) F(\lambda, \zeta) \, \dd\sigma(\zeta) \left| \hc(\lambda) \right|^{-2} \dd\lambda = 0.
  \end{equation*}
  The functions $\hat{g}$, for $g \in \D_\natural(B)$, form a subalgebra of the set of even continuous functions on $\R$ vanishing at infinity, which lies dense therein by the Stone-Weierstrass theorem.
  The inner integral thus vanishes for every $f$, up to a null set $N(f)$.

  The rest follows with the following approximation.
  Let $(r_k)$ be a sequence satisfying $0 < r_k < 1$, $\lim_{k\to \infty} r_k = 1$,
  and for each $k \in \N$ choose $\psi_k \in \D(B)$ with $\psi_k = 1$ on $r_kB$.
  Let $\mathcal{P}_\Q$ be the set of polynomials in $z_1, \bar{z}_1, \ldots, z_n, \bar{z}_n$ with rational coefficients.
  Then the set $\mathcal{M} = \left\{ \psi_k p : p \in \mathcal{P}_\Q, k \in \N \right\}$ is countable, furthermore $N = \bigcup_{f \in \mathcal{M}} N(f)$ is a countable union of null sets, therefore
  \begin{equation} \label{eq:orth_abz}
    \int_S \hat{f} (\lambda, \zeta) F(\lambda, \zeta)\, \dd\sigma(\zeta) = 0
  \end{equation}
  holds for all $f \in \mathcal{M}$ and $\lambda \notin N$.
  Now for each $f \in \D(B)$, we may choose $k \in \N$ so that $\supp f \subset \overline{r_k B}$.
  On this compact set $f$ can be approximated uniformly by functions $\psi_k p$, with $p \in \mathcal{P}_\Q$.
  Therefore \eqref{eq:orth_abz} holds for all $f \in \D(B)$ and $\lambda \in \C \setminus N$.
  By proposition \ref{thm:F_orth} we conclude $F(\lambda, \cdot) = 0$ for almost all $\lambda$, so $F = 0$.
\end{proof}

\section*{Acknowledgement}

I would like to thank Jens Wirth, Uwe Kähler, and Paula Cerejeiras for introducing me to this area of mathematics and for kindly providing valuable conversation and pieces of advice.

\newcommand{\etalchar}[1]{$^{#1}$}


\begin{thebibliography}{DFR{\etalchar{+}}08}

\bibitem[ABC96]{ahern_bru_casc}
Patrick Ahern, Joaquim Bruna, and Carme Cascante.
\newblock {$H^p$}-theory for generalized {$M$}-harmonic functions in the unit
  ball.
\newblock {\em Indiana Univ. Math. J.}, 45(1):103--135, 1996.

\bibitem[BW16]{beals_wong}
Richard Beals and Roderick Wong.
\newblock {\em Special Functions and Orthogonal Polynomials}.
\newblock Cambridge University Press, 2016.

\bibitem[DFR{\etalchar{+}}08]{da_fo_rau_st_te}
Stephan Dahlke, Massimo Fornasier, Holger Rauhut, Gabriele Steidl, and Gerd
  Teschke.
\newblock Generalized coorbit theory, {Banach} frames, and the relation to
  {$\alpha$}–modulation spaces.
\newblock {\em Proceedings of the London Mathematical Society}, 96(2):464--506,
  2008.

\bibitem[DST04a]{da_st_te1}
Stephan Dahlke, Gabriele Steidl, and Gerd Teschke.
\newblock Coorbit spaces and {Banach} frames on homogeneous spaces with
  applications to the sphere.
\newblock {\em Advances in Computational Mathematics}, 21(1-2):147--180, 2004.

\bibitem[DST04b]{da_st_te2}
Stephan Dahlke, Gabriele Steidl, and Gerd Teschke.
\newblock Weighted coorbit spaces and {Banach} frames on homogeneous spaces.
\newblock {\em Journal of Fourier Analysis and Applications}, 10(5):507--539,
  2004.

\bibitem[DST07]{da_st_te3}
Stephan Dahlke, Gabriele Steidl, and Gerd Teschke.
\newblock Frames and coorbit theory on homogeneous spaces with a special
  guidance on the sphere.
\newblock {\em Journal of Fourier Analysis and Applications}, 13(4):387--404,
  2007.

\bibitem[Fei81]{fei_char_hom}
Hans~G. Feichtinger.
\newblock A characterization of minimal homogeneous {Banach} spaces.
\newblock {\em Proceedings of the American Mathematical Society}, 81(1):55--61,
  1981.

\bibitem[Fer15]{ferreira_moebgyro}
Milton Ferreira.
\newblock Harmonic analysis on the {Möbius} gyrogroup.
\newblock {\em Journal of Fourier Analysis and Applications}, 21(2):281--317,
  2015.

\bibitem[FFP16]{fei_pes_fue}
Hans~G. Feichtinger, Hartmut Fuehr, and Isaac~Z. Pesenson.
\newblock Geometric space-frequency analysis on manifolds.
\newblock {\em Journal of Fourier Analysis and Applications}, 22(6):1294--1355,
  2016.

\bibitem[FG89]{fei_groe1}
Hans~G. Feichtinger and Karlheinz Gr{\"o}chenig.
\newblock Banach spaces related to integrable group representations and their
  atomic decompositions {I}.
\newblock {\em Journal of Functional Analysis}, 86(2):307--340, 1989.

\bibitem[Gr{\"o}01]{groechenig}
Karlheinz Gr{\"o}chenig.
\newblock {\em Foundations of Time-Frequency Analysis}.
\newblock Applied and Numerical Harmonic Analysis. Birkh{\"a}user Basel, 2001.

\bibitem[GS00]{groechenig_samarah}
Karlheinz Gr{\"o}chenig and Salti Samarah.
\newblock Nonlinear approximation with local fourier bases.
\newblock {\em Constructive Approximation}, 16(3):317--331, 2000.

\bibitem[Hel84]{helg_groups}
Sigurdur Helgason.
\newblock {\em Groups and Geometric Analysis}, volume~83 of {\em Mathematical
  Surveys and Monographs}.
\newblock American Mathematical Society, 1984.

\bibitem[Hel94]{helg_geom}
Sigurdur Helgason.
\newblock {\em Geometric Analysis on Symmetric Spaces}, volume~39 of {\em
  Mathematical surveys and monographs}.
\newblock American Mathematical Society, 1994.

\bibitem[LP09]{liu_peng}
Congwen Liu and Lizhong Peng.
\newblock Generalized {Helgason}-{Fourier} transforms associated to variants of
  the {Laplace}-{Beltrami} operators on the unit ball in $\mathbb{R}^n$.
\newblock {\em Indiana University Mathematics Journal}, 58(3):1457--1492, 2009.

\bibitem[Pes15]{pesenson_average_samp}
Isaac~Z. Pesenson.
\newblock Average sampling and space-frequency localized frames on bounded
  domains.
\newblock {\em Journal of Complexity}, 31(5):675--688, 2015.

\bibitem[Ros77]{rosenberg}
Jonathan Rosenberg.
\newblock A quick proof of {Harish-Chandra}'s {Plancherel} theorem for
  spherical functions on a semisimple {Lie} group.
\newblock {\em Proceedings of the American Mathematical Society}, 63:143--149,
  1977.

\bibitem[Rud80]{rudin_ball}
Walter Rudin.
\newblock {\em Function Theory in the Unit Ball of $\C^n$}.
\newblock Springer, 1980.

\bibitem[Zha92]{zhang}
Genkai Zhang.
\newblock A weighted {Plancherel} formula. {II}: The case of the ball.
\newblock {\em Studia Mathematica}, 102(2):103--120, 1992.

\end{thebibliography}

\vfill

email: \texttt{mathias.ionescu-tira@mathematik.uni-halle.de} 

date: \today
\end{document}